\documentclass[reqno,a4paper,11pt]{amsart}
\usepackage{amsmath, amssymb}
\usepackage{mathtools}
\usepackage{mathrsfs}
\usepackage[matrix, arrow]{xy}
\usepackage{tikz-cd}
\usepackage{amsthm}
\usepackage{graphicx}
\usepackage{fullpage}
\newtheorem{dfn}{Definition}[section]
\newtheorem{thm}[dfn]{Theorem}
\newtheorem{lemma}[dfn]{Lemma}
\newtheorem{prop}[dfn]{Proposition}
\newtheorem{cor}[dfn]{Corollary}
\newtheorem{rmk}[dfn]{Remark}

\newtheorem{con}[dfn]{Conjecture}
\theoremstyle{definition}

\theoremstyle{definition}
\newtheorem{exmp}[dfn]{Example}
\newcommand{\notmid}{\mathrel{\ooalign{$\mkern-5mu\not$\crcr$|$}}}

\title{Arithmetic period map and Complex multiplication  for cubic fourfolds}
\author{Rikuto Ito}
\address{Graduate School of Mathematics, Nagoya University, Furo-cho, Chikusa-ku, Nagoya, 464-8601, Japan}
\email{m21008a@math.nagoya-u.ac.jp}

\begin{document}
\begin{abstract} We construct an arithmetic period map for cubic fourfolds, in direct analogy with Rizov's work on K3 surfaces. For each $N\geq 1$, we introduce a Deligne-Mumford stack $\widetilde{\mathcal{C}^{[N]}}$ of cubic fourfolds with level structure and prove that the associated period map $j_{N}:\widetilde{\mathcal{C}^{[N]}}_{\mathbb{C}}\to {\rm Sh}_{K_{N}}(L)_{\mathbb{C}}$ is algebraic, \'etale, and descends to $\mathbb{Q}$ whenever $N$ is coprime to $2310$. As an application, we develop complex multiplication theory for cubic fourfolds and show that every cubic fourfold of CM type is  defined over an abelian extension of its reflex field. Moreover, using the CM theory  for rank-21 cubic fourfolds, we provide an alternative proof of the modularity of rank-21 cubic fourfolds established by Livn\'e. 
\end{abstract}
\maketitle

\section{Introduction}
The theory of arithmetic period maps plays a central role in understanding varieties whose Hodge structures are of K3 type. For polarized K3 surfaces, whose primitive cohomology carries a quadratic form of signature (2, 19), the Shimura variety associated with the special orthogonal group ${\rm SO}(2, 19)$ provides a natural period domain. Based on this viewpoint, Rizov \cite{Rizov10}, \cite{Rizov06} constructed the moduli space of polarized K3 surfaces with level structure, defined the corresponding period map, and proved that it descends to $\mathbb{Q}$, thereby establishing the complex multiplication theory for K3 surfaces.
\\\indent In a way independent of Rizov’s work, Madapusi Pera developed the arithmetic theory of the period map in mixed characteristic through the machinery of absolute Hodge cycles and the Kuga-Satake correspondence.  Moreover, in the same framework, he constructed an arithmetic period map for cubic fourfolds, and proved that it is étale (hence an open immersion under mild level assumptions) and that the Tate conjecture holds in odd characteristic (\cite[Theorem 5.14]{mp}). 
\\\indent In contrast, we follow Rizov’s moduli-geometric and complex-analytic approach, and develop
an independent arithmetic construction for cubic fourfolds without appealing to absolute Hodge cycles or the Kuga-Satake correspondence. Our method is purely moduli-theoretic: we explicitly build the Deligne-Mumford stack $\widetilde{\mathcal{C}^{[N]}}$ of cubic fourfolds with level structure,  and define the period morphism
\begin{align}
j_{N, \mathbb{Q}} : \widetilde{\mathcal{C}^{[N]}}_{\mathbb{Q}} \to {\rm Sh}_{K_{N}}(L)_{\mathbb{Q}}, \notag
\end{align}
proving that it is algebraic and an open immersion, and that its field of definition is $\mathbb{Q}$.
This gives a moduli-theoretic realization of the arithmetic period map for cubic fourfolds, complementary to Pera’s  approach via Hodge cycles. As an application, we obtain the corresponding complex multiplication theory (Corollaries 5.8 and 5.9).
\\\\\indent
Let $L$ be a lattice of signature $(n_{+}, n_{-})=(2, n)$ for a positive integer $n\geq 1$. The Shimura variety associated with the special orthogonal group ${\rm SO}(2, n)$ has been extensively studied by Deligne \cite{deligne}, Andr\'e \cite{and}, Kisin \cite{ki}, Madapusi Pera \cite{mp}, and Kim-Madapusi Pera \cite{kim}. For polarized K3 surfaces, whose  primitive cohomology carries a quadratic form of signature $(2, 19)$,  this Shimura variety provides  a natural period domain. It has been used to establish both  the complex multiplication theory for K3 surfaces (Rizov \cite[Corollary 3.9.2]{Rizov10}) and  the Tate conjecture for K3 surfaces in odd characteristic (Madapusi Pera \cite[Theorem 1]{mp}). 
\\\indent In this work, we turn  to $\mathit{cubic~ fourfolds}$, that is, smooth cubic hypersurfaces in $\mathbb{P}^{5}$. Their primitive cohomology has signature $(20, 2)$, and hence one can associate to them the  Shimura variety ${\rm Sh}(L)_{\mathbb{C}}$ corresponding to ${\rm SO}(2, 20)$. Our goal is to establish an analogue of the complex multiplication theory in this context (see Corollary \ref{cmtgen}). 
\\\indent Rizov's approach for K3 surfaces proceeds in three  steps: constructing the moduli space of K3 surfaces with  level structure (\cite{Rizov06}),  defining the period morphism (\cite{Rizov10}), and proving that its field of definition is $\mathbb{Q}$ (\cite{Rizov10}). We follow an analogous strategy for cubic fourfolds.
\\\indent  Let $\mathcal{C}$ denote the moduli stack of cubic fourfolds (see $\S$\ref{prec4}). It is a smooth Deligne-Mumford stack over $\mathbb{Z}$. For each integer $N\geq 1$, Javanpeykar and Loughran \cite{jal} constructed the moduli space $\mathcal{C}^{[N]}$ of cubic fourfolds with level-$N$ structures, which is a smooth Deligne-Mumford stack over $\mathbb{Z}[1/N]$, equipped with a canonical \'etale surjection $\mathcal{C}^{[N]}\to \mathcal{C}_{\mathbb{Z}[1/N]}$. 
\\\indent To define the natural domain of the period morphism for cubic fourfolds, we introduce a stack $\widetilde{\mathcal{C}^{[N]}}$ over $\mathbb{Z}[1/N]$ and construct  a canonical open immersion 
\begin{align}
   F: \widetilde{\mathcal{C}^{[N]}}\to \mathcal{C}^{[N]}. \notag 
\end{align}
This yields  our first main theorem (Theorem \ref{main1}).
 \begin{thm}\label{main1}Let $N\geq1$ be an integer. 
\\(1) The $\mathbb{Z}[1/N]$-stack $\widetilde{\mathcal{C}^{[N]}}$ is a Deligne-Mumford stack locally of finite type.
\\(2) If $N\geq 3$ and $N$ is coprime to $2310=2\cdot 3\cdot 5\cdot 7\cdot 11$, then  $\widetilde{\mathcal{C}^{[N]}}$ is a smooth affine scheme of finite type over $\mathbb{Z}[1/N]$. 
\end{thm}
\begin{rmk}(On the constant 2310).
\textnormal{The integer \(2310 = 2 \cdot 3 \cdot 5 \cdot 7 \cdot 11\) arises from the requirement that
the automorphism group of an object in the moduli stack of cubic fourfolds with level structure acts faithfully on the torsion \cite[Theorem 1.1]{jal}.
This ensures, as in Rizov’s construction for polarized K3 surfaces, that the moduli stack
\(\widetilde{\mathcal{C}^{[N]}}\) has trivial stabilizers when \(N\) is coprime to \(2310\),
so that it becomes an algebraic space (indeed, a smooth affine scheme).
The excluded primes \(2,3,5,7,11\) correspond to small torsion orders for which certain
automorphisms may act trivially on the associated cohomology lattice.
Hence the coprimality condition guarantees that the level structure kills all residual automorphisms,
making the stack representable by a scheme. (We thank B. Hassett  for raising this point, via private communication.)}
\end{rmk}
When $N\geq 3$ is coprime to $2310$, we use  $\widetilde{\mathcal{C}^{[N]}}$ to define the  $\mathit{period~ map}$
 \begin{align}
        j_{N, \mathbb{C}} : \widetilde{\mathcal{C}^{[N]}}_{\mathbb{C}}\to {\rm Sh}_{K_{N}}(L)_{\mathbb{C}}, \notag 
    \end{align}
    where $K_{N}\subset {\rm SO}(L)(\mathbb{A}_{f})$ is a compact open subgroup of finite index (see Definition \ref{levelfin}). 
    \begin{thm}\label{main2}
    The period map $j_{N, \mathbb{C}}$
    is algebraic and is an open immersion.
\end{thm}
Theorem \ref{main2} admits a geometric application to $\mathit{rank}$-$\mathit{21~cubic~ fourfolds}$ (Definition \ref{def21}),  which correspond to $\mathit{singular ~K3~ surfaces}$-that is, K3 surfaces with  maximal Picard number $20$. According to Shioda-Inose \cite{si}, the transcendental lattice of a rank-21 cubic fourfold is isomorphic to that of a singular K3 surface, and its reflex field is an imaginary quadratic extension over $\mathbb{Q}$ (see Proposition \ref{21fun}). Rizov proved that, for a given imaginary quaratic field $K$, singular K3 surfaces whose reflex field is $K$ is dense in the moduli space of polarized K3 surfaces. By combining Theorem \ref{main2} with the theory of special points on Shimura varieties, we obtain the same statement for cubic fourfolds. 
\begin{thm}\label{rank21ex}Let $K\subset \mathbb{C}$ be an imaginary quadratic field. Then rank-21 cubic fourfolds with its reflex field $K$ are Zariski-dense in $\mathcal{C}_\mathbb{C}$. In particular, there exist infinitely many rank-21 cubic fourfolds with its reflex field $K$.
\end{thm}
More strongly, we prove the density in  moduli space with a level structure (Theorem \ref{apl1}). 
\\\indent Yang \cite{yang} constructed for every imaginary quadratic field $K\subset\mathbb{C}$, there exists a rank-21 cubic fourfold whose reflex field $K$.  For singular K3 surfaces, the analogous statement was established by Piatetski-Shapiro and Shafarevich \cite{ps} and by  Taelman \cite{tae}. Yang's proof uses Fano correspondence for cubic fourfolds, and applies Rizov's CM-density for K3 surfaces (\cite[Proposition 3.8.1]{Rizov10}). On the other hand, \cite{ps} and \cite{tae} use the surjectivity of the global period map for K3 surfaces. 
We give a lattice-theoretic proof of the existence in Theorem \ref{apl1}. Since the global period map for cubic fourfolds is $not ~surjective$ (Laza \cite{laza}), our argument crucially depends on Theorem \ref{main2}. 
\\\indent Theorem \ref{rank21ex} is a key step toward  our final main result (Theorem \ref{main3}).
\\\\\indent Madapusi Pera established an arithmetic period map for cubic fourfolds by  absolute Hodge cycles. Yang \cite{yang}  further extended Pera’s  framework to irreducible symplectic varieties
of K$3^{[n]}$-type. 
\\\indent In contrast, our method is entirely moduli-theoretic, and offers a direct geometric realization of the arithmetic period map for cubic fourfolds, paralleling Rizov’s original treatment of K3 surfaces. Hence our results provide a complementary and fully self-contained realization of the arithmetic period map for Hodge structures of cubic fourfold type, \textit{from a moduli-theoretic and complex-analytic viewpoint\/}.
\begin{thm}\label{main3}
    The field of definition of the period morphism $j_{N, \mathbb{C}}$ is $\mathbb{Q}$.
\end{thm}
In general, a rank-21 cubic fourfold has an $associated ~K3 ~surface$ in the sense of Hassett \cite{has}.  By combining Rizov's complex multiplication theory for singular K3 surfaces (\cite[Theorem 3.4.2]{Rizov10}) with the study of transcendental motives by Bolognesi-Pedrini \cite{bo}, Murre \cite{mur}, Kahn-Murre-Pedrini \cite{ka}, and B\"ulles \cite{bu}, we deduce the complex multiplication theory for rank-21 cubic fourfolds (Theorem \ref{cmt}). This theory shows that every rank-21 cubic fourfold, regarded as a point of ${\rm Sh}_{K_{N}}(L)_{\mathbb{C}}$,  is compatible  with the action of $\sigma\in{\rm Aut}(\mathbb{C})$ on ${\rm Sh}_{K_{N}}(L)_{\mathbb{C}}$. 
\\\indent Rank-21 cubic fourfolds thus belong to the class of cubic fourfolds  $of ~CM~ type$, a notion generalizing that of abelian varieties with complex multiplication. Piatetski-Shapiro and Shafarevich \cite{ps} formulated the following conjecture concerning  the arithmetic of varieties of CM type: 
\begin{con}(\cite[Section 1, Conjecture]{ps})
    Varieties of CM type are defined over  number fields. 
\end{con}
For cubic fourfolds, we prove a stronger statement: any cubic fourfold of CM type is defined over an  $abelian ~extension$ of its reflex field (Corollary \ref{abelian ext}), as an application of Theorem \ref{main3}. Furthermore, Theorem \ref{main3} yields the  $complex ~multiplication ~theory ~for ~cubic ~fourfolds$ (Corollary \ref{cmtgen}).
\\\\\indent As an application of complex multiplication theory, we give a proof of the modularity of rank-21 cubic fourfolds defined over $\mathbb{Q}$, independent of Livn\'e's work \cite{livne}. Let $X/\mathbb{C}$ be a rank-21 cubic fourfold defined over $\mathbb{Q}$, and let $X_{0}$ be a model over $\mathbb{Q}$, i.e. $X=X_{0, \mathbb{C}}$. Let $l$ be a prime, and let $\overline{\mathbb{Q}}$ be the algebraic closure of $\mathbb{Q}$ in $\mathbb{C}$. Then, as explained in \S \ref{galoisrep} (see (\ref{galdef})), we obtain a 2-dimensional $l$-adic Galois representation
\begin{align}
    \rho_{l}:{\rm Gal}(\overline{\mathbb{Q}}/\mathbb{Q})\to {\rm GL}(T(X)_{\mathbb{Z}_{l}}),\notag
\end{align}
where $T(X)$ is the transcendental lattice of $X$. 
\\\indent Our main modularity result is the following:
\begin{thm}(Livn\'e \cite[Corollary 1.4]{livne}) \label{modular}
    Let $l$ be a prime. Let $X/\mathbb{C}$ be a rank-21 cubic fourfold defined over $\mathbb{Q}$, and let $X_{0}$ be a model over $\mathbb{Q}$. Then there exists a weight-$3$ CM newform $f$ such that 
    \begin{align}\label{L}
        L(\rho_{l}, s)=L(f, s-1). 
    \end{align}
\end{thm}
Livn\'e determines the weight of the CM newform associated to a given motive by Faltings'$p$-adic Hodge theory \cite{falt}. As an application of this method, he proves that singular K3 surfaces defined over $\mathbb{Q}$ correspond to CM newforms of weight 3 (\cite[Remark 1.6]{livne}).  \\\indent On the other hand, Shioda and Inose \cite{si} construct the Hecke character associated with a singular K3 surface using the theory of complex multipkication for elliptic curves. Livn\'e observes that, independently of Fartings' result, the weight of the corresponding CM newform can be determined by this Shioda-Inose construction (\cite[Remark 1.7]{livne}). 
\\\indent For every rank-21 cubic fourfold defined over $\mathbb{Q}$, we  construct the algebraic Hecke character $\psi$ of infinite type $(2, 0)$ by using an analogue of the latter approach. The character  $\psi$ arises from the complex multiplication theory developed in Theorem \ref{cmt}. By a result of Shimura \cite[Lemma 3]{sh}, one associate to $\psi$ a $q$-series (in fact, weight-3 cusp form) whose coefficients coincide with  the special values of $\psi$. Finally, the Galois representation $\rho_{f, l}$ attached to this form is constructed using Ribet's theory \cite{rib}, and comparison with $\rho_{l}$ yields the claimed equality (\ref{L}).
\\\\\indent Hulek and Kloosterman \cite{hul} proved the same functional equation (\ref{L}) for those rank-21 cubic fourfolds whose Fano variety of lines is birational to $S^{[2]}$ for some singular K3 surface $S/\mathbb{C}$.  However, \textit{there exist infinitely many rank-21 cubic fourfolds whose Fano variety is not birational to $S^{[2]}$ for any singular K3 surface $S$\/}. 
\\\indent Indeed, by Hassett \cite[Theorem 1.0.2]{has}, for some $d\in \mathbb{Z}$, there is an open immersion $\mathcal{C}^{mar}_{d}\hookrightarrow \mathcal{M}_{d}$, where $\mathcal{C}^{mar}_{d}$ is the moduli space of marked cubic fourfolds of discriminant $d$ and $\mathcal{M}_{d}$ is the moduli space of polarized K3 surfaces of degree $d$. Such a discriminant $d$ satisfies the numerical condition:
\begin{align}
    (\ast\ast) ~d\textnormal{~divides}~2n^{2}+2n+2\textnormal{~for some }n\in\mathbb{Z}. \notag 
\end{align}On the other hand, Addington \cite[Theorem 2]{add0} characterized the discriminants $d$ for which the Fano variety is birational to $S^{[2]}$ for some singular K3 surface $S$:
\begin{align}
    (\ast\ast\ast)~d\textnormal{ is of the form } \frac{2n^{2}+2n+2}{a^{2}}\textnormal{ for some }n, a\in \mathbb{Z}. \notag 
\end{align}
Clearly, $(\ast\ast\ast)$ implies $(\ast\ast)$, but $d=74$ does not satisfy the converse (Hassett \cite[\S 6, 6.1]{has}). By Rizov's CM density  \cite{Rizov10}, singular K3 surfaces are dense in $\mathcal{M}_{d}$ for all $d\geq 1$. Therefore, via the open immersion $\mathcal{C}^{mar}_{74}\hookrightarrow \mathcal{M}_{74}$, one obtains infinitely many rank-21 cubic fourfolds which do not arise from the birational-K3 construction.  
\\\\\indent   Sch\"utt \cite{schutt} asked  the following question:
\\\\\textbf{Question.} Which weight 3 CM newforms with rational Fourier coefficients  have geometric realizations in a smooth projective variety $X$ over $\mathbb{Q}$ with $h^{2, 0}(X)={\rm dim}H^{0}(X, \Omega^{2}_{X})=1$ ?
\\\\\indent Elkies and Sch\"utt \cite[Theorem 1]{es} proved that if we suppose extended Riemann Hypothesis (ERH) for odd real Dirichlet characters, then every
newform of weight 3 with rational coefficients is associated to a singular K3 surface over $\mathbb{Q}$. 
\\\indent In Theorem \ref{modular}, the CM newform $f$  has rational coefficients by Theorem \ref{hecke}. For the cubic fourfold $X$ appearing in Theorem \ref{modular}, the Fano variety $F(X)$ is an irreducible holomorphic symplectic variety by Beauvill-Donagi \cite{bedo},
and $F(X_{0})$ is a model of $F(X)$ over $\mathbb{Q}$. According to \cite[Proposition 2.2]{hul}, the Galois representations on cohomology satisfy 
\begin{align}
    H^{2}_{et}(F(X_{0})_{\overline{\mathbb{Q}}}, \mathbb{Q}_{l}(1))\cong H^{4}_{et}(X_{0, \overline{\mathbb{Q}}}, \mathbb{Q}_{l}(2)), \notag 
\end{align}
as ${\rm Gal}(\overline{\mathbb{Q}}/\mathbb{Q})$-modules. Hence Theorem \ref{modular} gives another geometric realization of weight-3 CM newforms with rational coefficients. 
\\\\\indent $\textbf{Acknowledgements}$  
\\\indent The author expresses  his deepest gratitude to his supervisor,  Sho Tanimoto, for suggesting the study of cubic fourfolds, for his invaluable technical advice, and for his constant encouragement and support throughout the author's graduate studies. His support has been essential in enabling the author to lead a fulfilling and productive research life.
\\\indent The author would like to express his sincere gratitude to his supervisor, Noriko Yui at Queen's University Canada, for her constant guidance and support. Her help has been invaluable not only in the technical aspects of this work, but also in many practical matters related to living in Canada.  
\\\indent He is  sincerely grateful to Brendan Hassett, Fumiya Okamura,  Ichiro Shimada, and Matthias Sch\"utt for many stimulating discussions and valuable suggestions. 
\\\indent The author was partially supported by JST FOREST program Grant number JPMJFR212Z, and partially supported by NSERC through Yui's individual Discovery Grant. 
\section{Preliminaries}\label{pre}
\subsection{Hodge structures of K3 type}
Let $\mathbb{S}$ denote the Deligne torus over $\mathbb{R}$, that is, the $\mathbb{R}$-algebraic group ${\rm Res}_{\mathbb{C}/\mathbb{R}}(\mathbb{G}_{m, \mathbb{C}})$, so that $\mathbb{S}(\mathbb{R})=\mathbb{C}^{\times}$ and $\mathbb{S}(\mathbb{C})=\mathbb{C}^{\times}\times \mathbb{C}^{\times}$. Let $V$ be a free $\mathbb{Z}$-module of finite rank (or equivalently, a finite dimensional $\mathbb{Q}$-vector space). A $\mathbb{Z}\textendash Hodge ~structure$ (resp. $\mathit{\mathbb{Q}\textendash Hodge ~structure}$)  of weight $n\in \mathbb{Z}$ on $V$ is an algebraic $\mathbb{R}$-representation $\mathbb{S}\to {\rm GL}(V_{\mathbb{R}})$
 such that the restriction to $\mathbb{G}_{m, \mathbb{R}}\subset \mathbb{S}$ acts via  $t \mapsto t^{-n}$. We often abbreviate  “$\mathbb{Z}$-Hodge structure” and “$\mathbb{Q}$-Hodge structure” as $\mathbb{Z}$-HS and $\mathbb{Q}$-HS, respectively.
 \\\indent If $h : \mathbb{S}\to {\rm GL}(V_{\mathbb{R}})$ defines a $\mathbb{Z}$-HS (or $\mathbb{Q}$-HS) of weight $n$, then we obtain the canonical decomposition 
 \begin{align}\label{dec}
     V_{\mathbb{C}}=\oplus_{p+q=n}V^{p, q}, 
 \end{align}
 where each $V^{p, q}\subset V_{\mathbb{C}}$ is a complex subspace satisfying $h_{\mathbb{R}}(z)v=z^{-p}\overline{z}^{-q}v$ for $v\in V^{p, q}$. The $\mathit{type}$ of a $\mathbb{Z}$-HS (resp.  $\mathbb{Q}$-HS) $h$ is the set of all pairs $(p, q)\in\mathbb{Z}^{2}$ such that $V^{p, q}\neq 0$. We say that  $h$ is of type $\{(p_{i}, q_{i})\}_{1\leq i\leq m}$ if the finite set $\{(p_{i}, q_{i})\}_{1\leq i\leq m}$ contains the type of $h$.
\begin{exmp}\label{00}
    Let $h : \mathbb{S}\to {\rm GL}(V_{\mathbb{R}})$ be a $\mathbb{Z}$-HS (resp. $\mathbb{Q}$-HS) of type $\{(-1, 1), (0, 0), (1, -1)\}$. Then, 
    \begin{align}
        V_{\mathbb{C}}=V^{1, -1}\oplus V^{0, 0}\oplus V^{-1, 1}.\notag 
    \end{align}
    In this decomposition,  $h_{\mathbb{R}}(z)$ acts on $V^{1, -1}$ by $z^{-1}\overline{z}$, acts trivially on $V^{0, 0}$, and on $V^{-1. 1}$ by $z\overline{z}^{-1}$. 
\end{exmp}
\begin{dfn}(Hodge structures of K3 type)
    A $\mathbb{Z}$-HS or $\mathbb{Q}$-HS $h$ of type $\{(-1, 1), (0, 0), (1, -1)\}$ is said to be \textbf{of K3 type} if ${\rm dim}_{\mathbb{C}}(V^{1, -1})=1$.
\end{dfn}

    Let $h : \mathbb{S}\to {\rm GL}(V_{\mathbb{R}})$ be a $\mathbb{Z}$-HS (resp. $\mathbb{Q}$-HS) of type $\{(-1, 1), (0, 0), (1, -1)\}$, and let $\psi$ be a polarization of $h$. Since the weight of $h$ is 0, we obtain a factorization:
    \begin{align}
    \mathbb{S}\to {\rm SO}(V, \psi)_{\mathbb{R}}\to {\rm GL}(V_{\mathbb{R}}). \notag 
    \end{align}
   When discussing a  polarized $\mathbb{Z}$-HS (resp. $\mathbb{Q}$-HS) $(h, \psi)$ of type $\{(-1, 1), (0, 0), (1, -1)\}$,  we will freely identify it with the real representation $ h : \mathbb{S}\to {\rm SO}(V, \psi)_{\mathbb{R}}$ instead of $h : \mathbb{S}\to {\rm GL}(V_{\mathbb{R}})$.

\subsection{Cubic fourfolds}\label{prec4}
\begin{dfn}
   Let $k$ be an arbitrary field. A cubic fourfold over $k$ is a smooth hypersurface of degree $3$ in $\mathbb{P}^{5}_{k}$.
\end{dfn}

Let $X$ be a  cubic fourfold over $\mathbb{C}$. Its Hodge diamond (see \cite[Chapter 1, 1.4]{Huy2})  is given by
     $$
    \begin{tikzcd}[column sep=5pt,row sep=1pt]
&&&&&1&&&&& &\quad&b_0=1\phantom{2.}\\
&&&&0&&0&&&& &\quad&b_1=0\phantom{2.}\\
&&0&&&1&&&0&& &\quad&b_2=1\phantom{2.}\\
&0&&&0&&0&&&0& &\quad&b_3=0\phantom{2.}\\
0&&1&&&21&&&1&&0 &\quad&b_4=23.\\
    \end{tikzcd}
     $$
    which shows that  $H^{4}(X, \mathbb{Z}(2))$ carries a $\mathbb{Z}$-HS of  K3 type.
     \\\\\indent  By the  Bloch–Srinivas principle \cite{block}, the cycle map 
     \begin{align}
         {\rm CH}^{2}(X)\to H^{4}(X, \mathbb{Z}(2)) \notag
     \end{align}
     is injective; denote its image by $A(X)$. Since $A(X)\subset H^{0, 0}\cap H^{4}(X, \mathbb{Z}(2))$, one has  $1\leq {\rm rk}A(X)\leq 21$. Moreover, the Hodge conjecture holds for  cubic fourfolds (\cite{zuc}), so that
     $A(X)=H^{0, 0}(X)\cap H^{4}(X, \mathbb{\mathbb{Z}}(2)).$ 
    \\\indent Let $X\subset\mathbb{P}^{5}_{\mathbb{C}}$ be a  cubic fourfold, and set $h:=c_{1}(\mathcal{O}_{\mathbb{P}^{5}}(1)|_{X})\in H^{2}(X, \mathbb{Z}(1))$.  Consider the intersection pairing 
     \begin{align}
       -(-, -) :   H^{4}(X, \mathbb{Z}(2))\times H^{4}(X, \mathbb{Z}(2))\to H^{8}(X, \mathbb{Z}(4))\xrightarrow{\sim} \mathbb{Z}.\notag
     \end{align}by \cite[Chapter 1, 1.1, Remark 1.4]{Huy2},  the $\mathbb{Z}$-module $H^{4}(X, \mathbb{Z}(2))$ is torsion-free. Furthermore, the lattice $(H^{4}(X, \mathbb{Z}(2)), -(-, -))$ is unimodular  of   signature $(n_{+}, n_{-})=(21, 2)$ (\cite[Chapter 1, 1.3 and 1.5]{Huy2}). Since $-(h^{2}, h^{2})=\int_{X}h^{4}=3$, the lattice $(H^{4}(X, \mathbb{Z}(2)), (-,-))$ is an odd unimodular lattice of  signature $(2, 21)$. By the classification of indefinite unimodular lattices (\cite[Chapter 14, Corollary 1.3 (ii)]{Huy}), 
     \begin{align}
         (H^{4}(X, \mathbb{Z}(2)), (-, -))\cong \langle 1\rangle^{\oplus 2}\oplus \langle -1\rangle^{\oplus 21}\cong E_{8}(-1)^{\oplus 2}\oplus U^{\oplus 2}\oplus \langle-1\rangle^{\oplus 3}. \notag  
     \end{align}
     From \cite[Chapter 14, Example 1.4]{Huy}, the orthogonal complement of $h^{2}$ is 
     \begin{align}
         \langle h^{2}\rangle^{\perp}\cong E_{8}(-1)^{\oplus 2}\oplus U^{\oplus 2} \oplus A_{2}(-1). \notag 
     \end{align}
     \begin{dfn} \label{notlat}Let $L_{0}:=\langle 1\rangle^{\oplus2}\oplus \langle-1\rangle^{\oplus 21} $, and choose a vector $v\in L_{0}$ with $(v, v)=-3$ such that its orthogonal complement in $L_{0}$ is $L:=\langle v\rangle^{\perp}\cong E_{8}(-1)^{\oplus 2}\oplus U^{\oplus 2} \oplus A_{2}(-1)$.  Set $V_{0}=L_{0}\otimes \mathbb{Q}$ and $V=L\otimes \mathbb{Q}$.
     \end{dfn}
     The $\mathbb{Z}$-HS on the primitive cohomology 
     \begin{align}
         PH^{4}(X, \mathbb{Z})(2)=\langle h^{2}\rangle^{\perp}\notag 
         \end{align}
        together with the bilinear form $(-, -)$, defines a polarized $\mathbb{Z}$-HS of K3 type, which we denote by $h_{X}$. 
        \\\indent The $\mathit{transcendental ~part}$ is defined by
        \begin{align}T(X):=A(X)^{\perp},\notag
        \end{align}
         the orthogonal complement of $A(X)$ in $(H^{4}(X, \mathbb{Z}(2)), (-, -))$.
     \begin{exmp}\label{exac4}Suppose that the base field is $\mathbb{C}$
     \\(1) (Fermat cubic fourfold)
The Fermat cubic fourfold is defined by
    \begin{align}
\label{Fermat}        x_{0}^{3}+x_{1}^{3}+x_{2}^{3}+x_{3}^{3}+x_{4}^{3}+x_{5}^{3}=0. 
    \end{align}
    Shioda \cite{fermat} showed that ${\rm rk}A(X)=21$. 
    \\\indent Fermat cubic fourfold admits automorphisms of  orders $2, 3,$ and $5$:
    \begin{align}
        (x_{0}:x_{1}:\cdots x_{5})&\mapsto (x_{1}:x_{0}:\cdots x_{5})\notag 
        \\ (x_{0}:\cdots: x_{5})&\mapsto (\zeta_{3}x_{0}:\cdots : x_{5})\notag 
        \\ (x_{0} : \cdots : x_{5})&\mapsto (x_{0} : x_{2}:\cdots : x_{5}: x_{1}),\notag 
    \end{align}
    where $\zeta_{3}$ is a primitive cube root of unity.
    \\(2) (Klein cubic fourfold)
The Klein cubic fourfold is a cubic fourfold is defined by
\begin{align}\label{klein}
    x_{0}^{2}x_{1}+x^{2}_{1}x_{2}+\cdots+ x^{2}_{5}x_{0}=0. 
\end{align}
González-Aguilera and Liendo \cite{vic} showed that it admits an automorphim $\varphi$ of order 7 (\cite[Remark 1.11]{vic}). After a linear change of coordinates diagonalizing $\varphi$, the equation (\ref{klein}) becomes 
    \begin{align}
        x^2_0x^4 + x^2_1x_2 + x_0x^2_2 + x^2_3x^5 + x_3x^2_4 + x_1x^2_5 + ax_0x_1x_3 + bx_2x_4x_5=0
    \end{align}
    for some $a, b\in\mathbb{C}$. Setting $\sigma:=(1, 2, 3, 4, 5, 6)$,  the automorphism is 
    \begin{align}
        \varphi^{1}_{7} : (x_{0}:\cdots :x_{5})\mapsto (\zeta^{\sigma_{0}}_{7}x_{0} : \cdots : \zeta_{7}^{\sigma_{5}}x_{5}),\notag 
    \end{align}
    where $\zeta_{7}$ is a primitive 7th root of unity (\cite[Theorem 2.8]{vic}). Furthermore, Billi, Grossi, and Marquand \cite[Theorem 3.3, Table 1]{bi} computed ${\rm rk}A(X)=19$, $T(X)\cong U(7)\oplus \begin{pmatrix}
        -2 & 1\\ 1& 10
    \end{pmatrix}$.
    \\ (3) (González-Aguilera and Liendo \cite[Theorem 2.8]{vic})
    Consider the cubic fourfold
    \begin{align} x^3_0 + x^2_1x_5 + x^2_2x_4 + x_2x^2_3 + x_1x^2_4 + x_3x^2_5=0. \notag 
    \end{align}
    Setting $\sigma:=(0, 1, 3, 4, 5, 9)$, and letting $\zeta_{11}$ be a primitive 11th root of unity, define 
    \begin{align}
        \varphi^{1}_{11} : (x_{0} :  \cdots :x_{5})\mapsto (\zeta_{11}^{\sigma_{0}}x_{0} : \cdots : \zeta_{11}^{\sigma_{5}}x_{5}). \notag 
    \end{align}
    Then $\varphi^{1}_{11}$ is a symplectic automorphism of the order 11. Billi, Grossi, and Marquand  \cite[Theorem 3.3, Table 1]{bi} proved that ${\rm rk}A(X)=21$, and that  the intersection matrix of the transcendental part is  $\begin{pmatrix}
        -22& 11\\ 11& -22
    \end{pmatrix}$. 
\end{exmp}
\begin{rmk}
     In general,  a cubic fourfold over an algebraically closed field admits  an automorphism of  prime order $p$ if and only if $p\in\{2, 3, 5, 7, 11\}$ (\cite[Theorem 2.6]{vic}).
\end{rmk}

  The $moduli ~stack ~\mathcal{C} ~ of ~cubic~ fourfolds$ is defined as follows.
     \begin{itemize}\item{ Objects:
     \\\indent Pairs $(\pi : X\to S, \lambda)$, where $\pi : X\to S $ is a smooth proper morphism of schemes and  $\lambda\in {\rm Pic}_{X/S}(S)$ is a polarization such that every geometric fiber $X_{\overline{s}}$ is isomorphic to a cubic fourfold $\subset \mathbb{P}^{5}_{k(\overline{s})}$, and the polarization  $\lambda_{\overline{s}}\in{\rm Pic}_{X_{\overline{s}}/k(\overline{s})}(k(\overline{s}))$ is the restrictoin of $\mathcal{O}_{\mathbb{P}^{5}_{k(\overline{s})}}(1)$.}
      \item{Morphism: 
      \\\indent For two objects $\mathcal{X}_{1} =(\pi_{1} : X_{1}\to S_{1}, \lambda_{1})$ and $\mathcal{X}_{2}=(\pi_{2} : X_{2}\to S_{2}, \lambda_{2})$, a morphism is a pair $(f_{S}, f)$, where $f_{S} : S_{1}\to S_{2}$ is a morhism of schemes and $f : X_{1}\to X_{2}\times_{S_{2}, f_{S}}S_{1}$ is an isomorphism over $S_{1}$ satisfying $f^{\ast}\lambda_{2}=\lambda_{1}$.}
      \end{itemize}
 Via the correspondence $(\pi : X\to S, \lambda)\mapsto S$, the category $\mathcal{C}$ acquires the structure of a category over the category of schemes $(Sch)$. By Benoist's thesis \cite[Proposition 2.3.1]{be}, if $H$ denotes the open subscheme of the Hilbert scheme of $\mathbb{P}^{5}_{\mathbb{Z}}$ parameterizing smooth cubic fourfolds,  then 
    \begin{align}
        \mathcal{C}\cong [{\rm PGL}_{6}\backslash H],\notag
    \end{align}
    which is a smooth, separated Deligne-Mumford stack of finite type over $\mathbb{Z}$.
    
   \subsection{The Shimura variety associated with the special orthogonal group ${\rm SO}(2, 20)$\label{sh}}
      Let $(V, (-, -))$ be the 22-dimensional quadratic space defined in Definition \ref{notlat}, and set $G\coloneq{\rm SO}(V, (-, -))$. Define $X_{L}$ to be the collection of all $\mathbb{Q}$-HSs $h : \mathbb{S}\to G_{\mathbb{R}}={\rm SO}(V_{\mathbb{R}})$ such that $\pm(-, -)$ is a polarization and the Hodge numbers satisfy $h^{-1, 1}=h^{1, -1}=1$, $h^{0, 0}=20$.  Then the pair $(G, X_{L})$ forms a $Shimura~ datum$.
 \\\indent For each $h\in X_{L}$, let $F_{h}$ denote the associated Hodge filtration. The map $h\mapsto F^{1}_{h}$ identifies $X_{L}$ with the set 
      \begin{align}
        \{\omega\in \mathbb{P}(V\otimes \mathbb{C})| ~(\omega, \omega)=0~ {\rm and}~(\omega, \overline{\omega})>0 \notag 
\},      \end{align}
which naturally endows $X_{L}$ with a  complex structure. 
\\\indent For a compact open subgroup $K\subset G(\mathbb{A}_{f})$, define the double coset space
\begin{align}
    {\rm Sh}_{K}(L)(\mathbb{C})=G(\mathbb{Q})\backslash X_{L}\times (G(\mathbb{A}_{f})/K). \notag 
\end{align}
\indent Let $G^{{\rm ad}} : = G/Z$ be the quotient by its center $Z$, and let  $G^{{\rm ad}}(\mathbb{R})^{+}$ denote the connected component of the identity in $G^{{\rm ad}}(\mathbb{R})$. Let $G(\mathbb{R})_{+}$ be the image of $G(\mathbb{R})$ in $G^{{\rm ad}}(\mathbb{R})$ lying inside $G^{{\rm ad}}(\mathbb{R})^{+}$, and set $G(\mathbb{Q})_{+}=G(\mathbb{Q})\cap G(\mathbb{R})_{+}$.
\\\indent Choose a connected component $X^{+}_{L}\subset X_{L}$, and fix a set $\mathcal{C}$ of representatives for the double coset space $G(\mathbb{Q})_{+}\backslash G(\mathbb{A}_{f}) / K$. Then the map 
\begin{align}
    \Gamma_{g}\backslash X^{+}_{L}\to {\rm Sh}_{K}(L)(\mathbb{C});~ [x]\mapsto [(x, g)] \notag 
\end{align} induces the canonical homeomorphism 
\begin{align}\label{shdec}
    \coprod_{g\in \mathcal{C}}\Gamma_{g}\backslash X^{+}_{L}\to {\rm Sh}_{K}(L)(\mathbb{C}),
\end{align}
where $\Gamma_{g}:=gKg^{-1}\cap G(\mathbb{Q})_{+}$ is a subgroup of $G(\mathbb{Q})_{+}$ (cf. \cite[Lemma 5.13]{milne}). Since each quotient $\Gamma_{g}\backslash X^{+}$ is a quasi-projective variety, the space ${\rm Sh}_{K}(L)(\mathbb{C})$ acquires the structure of a quasi-projective scheme over $\mathbb{C}$. If $K$ is sufficiently small, the group $\Gamma_{g}$ is torsion-free for all $g\in \mathcal{C}$ (\cite[Proposition 3.5]{milne}). Hence the algebraic structure on ${\rm Sh}_{K}(L)(\mathbb{C})$ is uniquely determined (\cite[Theorem 3.10]{bor}), and  we denote this unique $\mathbb{C}$-scheme by ${\rm Sh}_{K}(L)_{\mathbb{C}}$.
\subsection{Canonical models of Shimura varieties}\label{rx}Let $(G, X)$ be a Shimura datum, and consider the set
\begin{align}
    \mathscr{C}(\mathbb{C}):=G(\mathbb{C})\backslash{\rm Hom}(\mathbb{G}_{m, \mathbb{C}}, G_{\mathbb{C}}).\notag
\end{align}
Since both $\mathbb{G}_{m, \mathbb{C}}$ and $G_{\mathbb{C}}$ are defined over $\mathbb{Q}$, the  group ${\rm Aut}(\mathbb{C})$ acts naturally on ${\rm Hom}(\mathbb{G}_{m, \mathbb{C}}, G_{\mathbb{C}})$, and hence  induces an action on $\mathscr{C}(\mathbb{C})$.
Let $\mu : \mathbb{G}_{m, \mathbb{C}}\to \mathbb{S}_{\mathbb{C}}$ be the morphism given by
\begin{align}
    \mathbb{G}_{m, \mathbb{C}}\to \mathbb{G}_{m, \mathbb{C}}\times \mathbb{G}_{m, \mathbb{C}} ;~z\mapsto(z, 1).\notag 
\end{align} For each $h\in X$, define
\begin{align}
    \mu_{h}=h\circ \mu : \mathbb{G}_{m, \mathbb{C}} \to G_{\mathbb{C}}. \notag 
\end{align}If $h_{0}\in X$ is another point, then $\mu_h$ and $\mu_{h_{0}}$ are conjugate under $G(\mathbb{C})$. The corresponding conjugacy class in $\mathscr{C}(\mathbb{C})$ is denoted by $c(X)$.
\begin{dfn}(Reflex field)
    The \textbf{reflex field} $E(G, X)$ is defined as the subfield of $\mathbb{C}$ such that the stabilizer  of $c(X)$ under the action of ${\rm Aut}(\mathbb{C})$ on $\mathcal{C}(\mathbb{C})$ is precisely ${\rm Aut}(\mathbb{C}/E(G, X))$.
\end{dfn}
\begin{exmp}\label{sp}
    (1) Let $T$ be a $\mathbb{Q}$-torus and $h : \mathbb{S}\to T_{\mathbb{R}}$ an $\mathbb{R}$-algebraic representation. Then the pair $(T, \{h\})$ forms a Shimura datum, and the associated cocharacter $\mu_{h}$ is defined over $E(T, \{h\})$.
    \\(2) For the Shimura datum $(G, X_{L})$ introduced in $\S$\ref{sh},  Andr\'e \cite{and} proved that the reflex field of $(G, X_{L})$ is $\mathbb{Q}$.
\end{exmp}
\begin{dfn}(Special points) Let $(G, X)$ be a Shimura datum.
        A point $h\in X$ is called \textbf{special} if there exists a $\mathbb{Q}$-torus $T\subset G$ such that the morphism $h : \mathbb{S}\to G_{\mathbb{R}}$ factors through an $\mathbb{R}$-morphism $\mathbb{S}\to T_{\mathbb{R}}$. In this case, we call $(T, h)$ a \textbf{special pair} in $(G, X)$. The reflex field $E(T, h)$ of such a pair is independent of the choice of $T$; we therefore write simply $E(h)$.
       
    \end{dfn}
  For any special pair $(T, h)$  of a Shimura datum $(G, X)$, the cocharacter $\mu_{h}$ is defined over $E(h)$ (see Example \ref{sp}). Composing $\mu_{h}: \mathbb{G}_{m, E(h)}\to T_{E(h)}$ with the norm map ${\rm Res}_{E(h)/\mathbb{Q}}(T_{E(h)})\to T$, we have a morphism of algebraic groups
    \begin{align}
        r(T, \mu_h) : {\rm Res}_{E(h)/\mathbb{Q}}(\mathbb{G}_{m, E(h)})\to T. \notag 
    \end{align}
     Passing to adelic points yields a continuous homomorphism  
        \begin{align}
        \mathbb{A}_{E(h)}^{\times}= {\rm Res}_{E(h)/\mathbb{Q}}(\mathbb{G}_{m, E(h)})(\mathbb{A}_{\mathbb{Q}})\xrightarrow{r(T, \mu_h)} T(\mathbb{A}_{\mathbb{Q}})\xrightarrow{{\rm proj}}T(\mathbb{A}_{f}).\notag
    \end{align}
\begin{dfn}(canonical models)
    Let $(G, X_{L})$ be the Shimura datum introduced in $\S$\ref{sh}, and let $K\subset G(\mathbb{A}_{f})$ be a compact open subgroup. A \textbf{canonical model} of ${\rm Sh}_{K}(L)_{\mathbb{C}}$ is an $E(G, X_{L})$-scheme $S$ equipped with an isomorphism 
   \begin{align}\label{canmod}
        S_{\mathbb{C}}\cong {\rm Sh}_{K}(L)_{\mathbb{C}}
    \end{align}
    such that the following condition holds: 
    \\ \indent For every special pair $(T, h)$ in $(G, X_{L})$ and every $a\in G(\mathbb{A}_{f})$,  the point $[h, a]$ is $E(h)^{ab}$-rational under the isomorphism (\ref{canmod}), and for all $\sigma\in{\rm Gal}(E(h)^{ab}/E(h))$ and $s\in \mathbb{A}^{\times}_{E(h)}$ satisfying ${\rm art}_{E(h)}(s)=\sigma$, one has 
    \begin{align}
        \sigma[h, a]=[h, r_{h}(s)a],
    \end{align},
    where $r_{h}=r(T, \mu_{h})$.
    \end{dfn}
A canonical model of ${\rm Sh}_{K}(G, X)_{\mathbb{C}}$ always exists and is unique up to a unique isomorphism (\cite[Section 14 and Theorem 13.7]{milne}).
\subsection{The endomorphism ring of $\mathbb{Q}$-Hodge structure of K3 type}
 Let $T$ be an $irreducible$ $\mathbb{Q}$-HS of K3 type, and let $K(T):={\rm End}_{{\rm Hdg}}(T)$ be its ring of Hodge endomorphisms. This  is a finite-dimensional $\mathbb{Q}$-algebra. Let  $\omega$ be a generator of $T^{1, -1}$. Then we obtain a  homomorphism of $\mathbb{Q}$-algebras  
    \begin{align}\label{dif}
        \varepsilon : K(T)\to \mathbb{C}\cong \mathbb{C}\omega=T^{1, -1} , ~a\mapsto \varepsilon(a),
    \end{align}
    defined by the rule $a(\omega)=\varepsilon(a)\omega$. Note that the morphism $\varepsilon$ is independ of the choice of $\omega$. 
    \begin{prop}(\cite[Chapter 3, Corollary 3.6]{Huy})
       The homomorphism $\varepsilon$ in (\ref{dif}) is injective, and $K(T)$ is a number field. 
    \end{prop}
    If the $\mathbb{Q}$-HS $T$ is $polarizable$, then the following stronger result holds.
    \begin{thm}(Zarhin \cite[Theorem 1.5.1]{Zarhin}) 
        Let  $T$ be an irreducible, polarizable $\mathbb{Q}$-HS of K3 type. Then the field $K(T)$ is either totally real or a CM field. 
    \end{thm}
   \begin{exmp}
       (1) Let $S$ be a K3 surface over $\mathbb{C}$. Then the transcendental part $T(S)$ is an irreducible , polarized $\mathbb{Q}$-HS of K3 type. 
       \\(2) Let $X$ be a cubic fourfold over $\mathbb{C}$. Then its transcendental part $T(X)$ is likewise an irreducible, polarized $\mathbb{Q}$-HS
 of K3 type.    \end{exmp}
 
 Let $h : \mathbb{S} \to {\rm GL}(V_{\mathbb{R}})$ be a $\mathbb{Q}$-HS. The $Mumford\textendash Tate~ group$ ${\rm MT}(h)$ is the smallest $\mathbb{Q}$-algebraic subgroup of ${\rm GL}(V)$ such that $h_{\mathbb{R}}(\mathbb{S}(\mathbb{R}))\subset {\rm MT}(h)(\mathbb{R})$. 
\begin{exmp}\label{tri}
    Let $h : \mathbb{S}\to {\rm GL}(V_{\mathbb{R}})$ be a $\mathbb{Q}$-HS of weight $0$, and  suppose that $V^{1, -1}=\{0\}$. Then ${\rm MT}(h)$ is the trivial subgroup of ${\rm GL}(V)$, since $h_{\mathbb{R}}(\mathbb{S}(\mathbb{R}))$ acts trivially.
    \\\indent In particular, for a cubic fourfold $X/\mathbb{C}$, the Mumford-Tate group ${\rm MT}(h_{X})$ coincides with that of its restriction $h_{X}|_{T(X)} : \mathbb{S}\to {\rm GL}(T(X)_{\mathbb{R}})$ 
\end{exmp}
\begin{thm}\label{zrfun}
    Let $(T, \psi)$ be an irreducible polarized $\mathbb{Q}$-HS of K3 type with endomorphism field $K=K(T)$. Then the Mumford-Tate group of $T$  is 
  the group of $K$-linear special isometries: 
  \begin{align}{\rm MT}(T)={\rm SO}_{K}(T)\subset {\rm SO}(T, \psi).\notag 
  \end{align} In particular ${\rm MT}(T)$ is commutative if and only if ${\rm dim}_{K}T=1$, in which case  $K$ is a CM field.
\end{thm}
    \begin{dfn} (Cubic fourfolds over $\mathbb{C}$ of CM type)
 Let $X$ be a cubic fourfold over $\mathbb{C}$. We say that $X$ is of \textbf{CM type} if the Mumford-Tate group  of $h_{X}$ is commutative. 
 \end{dfn}
 By Theorem \ref{zrfun} and Example \ref{sp} (1),  whenever  a cubic fourfold $X$ is of CM type, the pair $({\rm MT}(h_{X}), \{h_{X}\})$ forms s a Shimura datum. 
 \\\indent Valloni described the  reflex field if such cubic fourfolds as follows.
 \begin{prop}(\cite[Proposition 2.2]{val})
    Let $X$ be a cubic fourfold of CM type. With the notation of (\ref{dif}), its reflex field is $\varepsilon(K(T(X))\subset \mathbb{C}$.
 \end{prop}
  In particular, for any isometry $a : H^{4}(X, \mathbb{Z}(2))\to L_{0}$ sending $h^{2}$ to $v$, the reflex field of $\mathbb{Q}$-HS $a\circ h_{X}\circ a^{-1}$ is also $\varepsilon(K(T(X))\subset \mathbb{C}$.
  \\\indent Finally, we record a lemma that will be used in the proof of Theorem \ref{main3}.
 \begin{lemma}\label{conja}
     Let $X$ be a cubic fourfold over $\mathbb{C}$ of CM type, and let $h_{X}$ be its polarized $\mathbb{Q}$-HS with  reflex field $E\subset \mathbb{C}$. If $a : H^{4}(X, \mathbb{Z}(2))\to L_{0}$ is an isometry sending $h^{2}$ to $v$, then for all $s\in \mathbb{A}^{\times}_{E}$ we have the equality in $V_{\mathbb{A}_{f}}$:
     \begin{align}
        a\circ r_{X}(s)\circ a^{-1}=r_{(a\circ h_{X}\circ a^{-1})}(s), \notag 
     \end{align}
 where $r_{X}$ denotes $r_{h_{X}}$.
 \end{lemma}
 \begin{proof}
     Set $\mu_{X}:=\mu_{h_{X}}$. Then the diagram 
     \[\xymatrix{{\rm MT}(h_{X})_{E}\ar[rr]^{a\cdot (-)\cdot a^{-1}}& & {\rm MT}(a\circ h_{X}\circ a^{-1})_{E}\\
     & \mathbb{G}_{m, E}\ar[lu]^{\mu_{X}}\ar[ru]_{\mu_{a\circ h_{X}\circ a^{-1}}}&}\]
     is commutative. Applying the restriction of scalars ${\rm Res}_{E/\mathbb{Q}}$  and evaluating on adelic points yields the desired equality.
 \end{proof}
  \subsection{Chow-K\"unneth decomposition of K3 surfaces and cubic fourfolds}
  We recall here some basic facts about $Chow ~motives$ of smooth projective varieties over $\mathbb{C}$.  Let  $\mathsf{SmProj}(\mathbb{C})$ denote the category of smooth projective varieties  over $\mathbb{C}$.
  \\\indent The category of $Chow ~motives$, denoted $\mathsf{Mot}_{{\rm rat}}(\mathbb{C})$, has objects of the form $(X, p, m)$, where $X$ is a smooth projective variety over $\mathbb{C}$, $p\in {\rm CH}^{{\rm dim}(X)}(X\times X)_{\mathbb{Q}}$ is a projector (i.e. $p\circ p=p$), and $m\in \mathbb{Z}$. Morphisms are defined by 
  \begin{align}
      {\rm Hom}((X, p, m), (Y, q, n))=q\circ {\rm CH}^{{\rm dim}(X)+n-m}(X\times Y)_{\mathbb{Q}}\circ p.\notag 
  \end{align}
  For each smooth projective variety $X/\mathbb{C}$,  the triple $h_{{\rm rat}}(X):=(X, \Delta_{X}, 0)$ defines an object of $\mathsf{Mot}_{{\rm rat}}(\mathbb{C})$. If $f : X\to Y$ is a morphism of complex varieties, then the transpose of its graph $^{\mathsf{T}}\Gamma_{f} $ defines a morphism  $h_{{\rm rat}}(Y)$ to $h_{{\rm rat}}(X)$. Thus one obtain a natural contravariant functor 
  \begin{align}
      h : \mathsf{SmProj}(\mathbb{C})\to \mathsf{Mot}_{{\rm rat}}(\mathbb{C}). \notag 
  \end{align} 
  \\\indent Let $\sigma\in {\rm Aut}(\mathbb{C})$. For any $X\in \mathsf{smProj}(\mathbb{C})$, define the variety $X^{\sigma}\in \mathsf{smProj}(\mathbb{C})$ by the Cartesian diagram 
  \[\xymatrix{X^{\sigma}\ar[r]\ar[d]& X\ar[d]\\ {\rm Spec}(\mathbb{C})\ar[r]^{\sigma^{\ast}} & {\rm Spec}(\mathbb{C}).}\]
  The projection $X^{\sigma}\to X$ is denoted  $\sigma^{\ast}$. 
  \\\indent The automorphism $\sigma$ induces a functor 
  \begin{align}
      \sigma :\mathsf{Mot}_{{\rm rat}}(\mathbb{C})&\to \mathsf{Mot}_{{\rm rat}}(\mathbb{C})
      \\ (X, p, m)&\mapsto  (X^{\sigma}, p^{\sigma}, m)\notag 
      \\ (f : (X, p, m)\to (Y, q, n))&\mapsto (f^{\sigma} : (X^{\sigma}, p^{\sigma}, m)\to (Y^{\sigma}, q^{\sigma}, n)).\notag 
  \end{align}
 
  \indent Fix a Weil cohomology theory $H^{\ast}$ with coefficients in a field of characteristic 0, and let $X$ be a smooth projective variety over $\mathbb{C}$. Denote by $\gamma_{X} : {\rm CH}^{i}(X)\to H^{i}(X)(i)$ the cycle map.
 \begin{dfn}(Chow-K\"unneth decomposition)
     A \textbf{Chow-K\"unneth decomposition} of $X$ is a collection of orthogonal projectors  $p^{0},\cdots, p^{2d}\in {\rm CH}^{d}(X\times X)_{\mathbb{Q}}$ such that $\gamma_{X}(p^{i})$ is the $i$-th K\"unneth projector and $\sum p^{i}=\Delta_{X}$.
 \end{dfn} 
 \begin{thm}A Chow-K\"unneth decomposition exists for every surface over $\mathbb{C}$ (\cite[Theorem 3]{mur}) and for every cubic fourfold over $\mathbb{C}$ (\cite[Section 2]{bo}).   
 \end{thm}
 \indent Given such a decomposition, define $h^{i}(X):=(X, p_{i}, 0)$. Then one has  
 \begin{align}
     h(X)\cong \bigoplus^{2d}_{i=0}h^{i}(X),
 \end{align} 
 and this splitting satisfies 
 \begin{align}
     H^{j}(h^{i}(X))=\begin{cases}
         H^{i}(X) &(i=j)
         \\ 0 &(i\neq j).
     \end{cases}\notag 
 \end{align}If $X$ is a K3 surface over $\mathbb{C}$, then the only nontrivial part of $h^{i}(X)$ is $h^{2}(X)$. 
 \begin{thm}(\cite[Proposition 14.2.3]{ka})\label{k3mot}
     Let $S$ be a K3 surface over $\mathbb{C}$. There exists a unique decomposition
     \begin{align}
         p^{2}=p^{2}_{alg}+p^{2}_{tr}\in {\rm CH}^{2}(S\times S)_{\mathbb{Q}} \notag 
     \end{align}
     which induces 
     \begin{align}
         h^{2}(S)=h^{2}_{alg}(S)\oplus t(S),\notag 
     \end{align}
     where $h^{2}_{alg}(S):=(S, p^{2}_{alg}, 0)$ and $t(S):=(S, p^{2}_{tr}, 0)$. Moreover, the motive $h^{2}_{alg}(S)$ is isomorphic to $\underline{{\rm NS}}(S)_{\mathbb{Q}}\otimes \mathbb{L}$, where $\underline{{\rm NS}}(S)_{\mathbb{Q}}$ denotes the Artin motive associated with the N\'eron-Severi group ${\rm NS}(S)_{\mathbb{Q}}$.
 \end{thm}
Thus
 \begin{align}\label{deck3}
     H^{2}(S)=H^{2}(h^{2}_{alg}(S))\oplus H^{2}(t(S))={\rm NS}(S)_{\mathbb{Q}}(-1)\oplus T(S)_{\mathbb{Q}}(-1),
 \end{align} 
 where $T(S)$ denotes the transcendental lattice of $S$.
\\ \indent If $X$ is a cubic fourfold, an analogous decomposition holds. 
 \begin{thm}(\cite[Section 2]{bo})\label{c4mot}
     There exists a unique decomposition 
     \begin{align}
         p^{4}=p^{4}_{alg}+p^{4}_{tr}\in {\rm CH}^{4}(X\times X)_{\mathbb{Q}} \notag 
     \end{align}
     which induces 
     \begin{align}
         h^{4}(X)=h^{4}_{alg}(X)\oplus t(X), \notag 
     \end{align}
     where $h^{4}_{akg}(X):=(X, p^{4}_{alg}, 0)$ and $t(X):=(X, p^{4}_{tr}, 0)$.  Moreover, the motive $h^{4}_{alg}(X)$ is isomorphic to $\underline{A}(X)_{\mathbb{Q}}\otimes \mathbb{L}^{\otimes 2}$, where   $\underline{A}(X)_{\mathbb{Q}}$ is the Artin motive associated with the group $A(X)_{\mathbb{Q}}$.
 \end{thm}
Consequently,
 \begin{align}\label{decc4}
     H^{4}(X)=H^{4}(h^{4}_{alg}(X))\oplus H^{4}(t(X))=A(X)_{\mathbb{Q}}(-2)\oplus T(X)_{\mathbb{Q}}(-2).
 \end{align}

  \section{Moduli spaces with  level structure; the proof of Theorem \ref{main1}}\label{moduli main1}
     \subsection{ Level Structures on Cubic Fourfolds}
     Let $N\geq 1$ be an integer.
     \begin{dfn}(Javanpeykar-Loughran \cite{jal}) Let $S$ be a scheme over $\mathbb{Z}[1/N]$,  and let $(\pi : X\to S, \lambda)$ be an object of $\mathcal{C}_{\mathbb{Z}[1/N_{\mathscr{B}}]}$. 
         A $level\textendash N~structure$ on $(\pi, \lambda)$ is an isomorphism of \'etale sheaves on $S$
         \begin{align}\alpha : R^{4}_{et}\pi_{\ast}\mathbb{Z}/N\mathbb{Z}(2)\to (L_{0, \mathbb{Z}/N\mathbb{Z}})_{S}, \notag 
         \end{align}
         where the right-hand side denotes the constant \'etale sheaf on $S$.
         \\\indent Let $\mathcal{C}^{[N]}$ denote the category fibred in groupoids (CFG) over $\mathbb{Z}[1/N]$ whose objects are triples $(\pi:X\to S, \lambda, \alpha)$. A morphism between two such objects $(\pi_{i} : X_{i}\to S_{i}, \lambda_{i}, \alpha_{i}) (i=1, 2)$ is a triple $(f_{S}, f, \varphi)$ consisting of 
         \begin{itemize}
             \item{a morphism  $f_{S} : S_{1}\to S_{2}$ of schemes,}
             \item{an isomorphism $f : X_{1}\to X_{2}\times_{S_{2}, f_{S}} S{1}$ over $S_{1}$ with $f^{\ast}\lambda_{2}=\lambda_{1}$,}
             \item{an isomorphism of \'etale sheaves on $S_{1}$
         \[\varphi : R^{4}_{et}\pi_{1\ast}\mathbb{Z}/N\mathbb{Z}(2)\to f^{\ast}_{S}R^{4}_{et}\pi_{2\ast}\mathbb{Z}/N\mathbb{Z}(2)\]
         maiking the diagram commute:
        \[ \xymatrix{ R^{4}_{et}\pi_{1\ast}\mathbb{Z}/N\mathbb{Z}(2)\ar[rr]^{\varphi}\ar[rd]_{\alpha_{1}}&&f_{S}^{\ast}R^{4}_{et}\pi_{2 \ast}\mathbb{Z}/N\mathbb{Z}(2)\ar[ld]^{f_{S}^{\ast}\alpha_{2}}\\ & (L_{0, \mathbb{Z}/N\mathbb{Z}})_{S_{1}}.&}\]}
        \end{itemize}
     \end{dfn}
     By \cite[Theorem 3.2, (2)]{jal}, the CFG $\mathcal{C}^{[N]}$ is smooth, of finite type,  separated,  and Deligne-Mumford over $\mathbb{Z}[1/N]$.
     \\\\\indent We now prepare  a subcategory $\widetilde{\mathcal{C}^{[N]}}$ of $\mathcal{C}^{[N]}$.  Let $\mathscr{B}:=\{p_{1}, \cdots, p_{r}\}$ be the set of prime divisors of $N$ and write  $\mathbb{Z}_{\mathscr{B}}:=\prod_{p\in \mathscr{B}}\mathbb{Z}_{p}$.  
     Let $S$ be a scheme over $\mathbb{Z}[1/N]$ and  $(\pi : X\to S, \lambda)$ be an object of $\mathcal{C}_{\mathbb{Z}[1/N_{\mathscr{B}}]}$. The Kummer short exact sequence  of \'etale sheaves on $X$ (see \cite[Example 10.5.2 (2)]{ss})
    \begin{align}
        1\to \mu_{N}\to \mathbb{G}_{m}\xrightarrow{\times N} \mathbb{G}_{m}\to 1\notag 
    \end{align}
    induces a long exact sequence 
    \begin{align}
        R^{1}_{et}\pi_{\ast}\mathbb{Z}/N\mathbb{Z}(1)\to R^{1}_{et}\pi_{\ast}\mathbb{G}_{m}\to R^{1}_{et}\pi_{\ast}\mathbb{G}_{m}\to R^{2}_{et}\pi_{\ast}\mathbb{Z}/N\mathbb{Z}(1). \notag 
    \end{align}
    Since the Picard group of cubic fourfold is torsion-free in all characteristics (\cite[Corollary 1.9]{Huy2}), we have  $R^{1}_{et}\pi_{\ast}\mathbb{Z}/n\mathbb{Z}(1)=0$. Hence there is an exact sequence
    \begin{align}
        0\to {\rm Pic}_{X/S}\to R^{2}_{et}\pi_{\ast}\mathbb{Z}/n\mathbb{Z}(1). 
    \end{align}
    Passing to  $\mathbb{Z}_{\mathscr{B}}$ gives
    \begin{align}\label{highcyc}
        0\to {\rm Pic}_{X/S}\to R^{2}_{et}\pi_{\ast}\mathbb{Z}_{\mathscr{B}}(1).
    \end{align}
   
    Let $h$ denote the image of $\lambda$ in (\ref{highcyc}). Let $h^{2}$ be its square under the cup product 
    \begin{align}
        -(-, -) : R^{2}\pi_{\ast}\mathbb{Z}_{\mathscr{B}}(1)\times R^{2}\pi_{\ast}\mathbb{Z}_{\mathscr{B}}(1)\to R^{4}\pi_{\ast}\mathbb{Z}_{\mathscr{B}}(2).\notag 
        \end{align}
     Define the sheaf of primitive cohomology $P^{4}_{et}\pi_{\ast}\mathbb{Z}_{\mathscr{B}}(2)$ to be the orthogonal complement of $h^{2}$  with respect to the cup product $R^{4}\pi_{\ast}\mathbb{Z}_{\mathscr{B}}(2)\times R^{4}\pi_{\ast}\mathbb{Z}_{\mathscr{B}}(2)\to R^{8}\pi_{\ast}\mathbb{Z}_{\mathscr{B}}(4)\cong  \mathbb{Z}_{\mathscr{B}}$. 
     \begin{dfn}The category $\widetilde{\mathcal{C}^{[N]}}$ is the full subcategory of $\mathcal{C}^{[N]}$ whose objects are triples $(\pi : X\to S, \lambda, \alpha)$  such that $\alpha$ is an isometry of \'etale sheaves
     \begin{align}
         \alpha :(R^{4}_{et}\pi_{\ast}\mathbb{Z}/N\mathbb{Z}(2), (-, -))\to ((L_{0, \mathbb{Z}/N\mathbb{Z}})_{S}, (-, -)_{\mathbb{Z}/N\mathbb{Z}})\notag 
     \end{align}
     mapping $h^{2}$ to $v$.
     \end{dfn}
      From the definition  
       there is a canonical fully faithful functor 
    \begin{align}\label{Fd}
        F : \widetilde{\mathcal{C}^{[N]}}\to \mathcal{C}^{[N]}.
    \end{align}
    \indent Following Rizov \cite{Rizov06}, we describe the level-$N$ structures when $S$ is connected. 
    \begin{dfn}\label{levelfin} For each integer $N\geq 1$, set $K_{N}$
     \begin{align}
         K_{N}\coloneq\{\gamma\in {\rm SO}(V_{0})(\hat{\mathbb{Z}}) ~| ~\gamma\equiv 1 ~({\rm mod}~N),~\gamma(v)=v.\}, \notag 
     \end{align}
         and $K_{N, \mathscr{B}}:=\prod_{p\in \mathscr{B}} K_{N, p}$,  where $K_{N, p}$ is the $p$-component.
     \end{dfn}
      Assume $S$ is a connected $\mathbb{Z}[1/N]$-scheme.  Fix a geometric point $\overline{b} : {\rm Spec}(k(\overline{b}))\to S$ and  set $H^{4}(\overline{b})\coloneq\overline{b}^{\ast}R^{4}_{et}\pi_{\ast}\mathbb{Z}_{\mathscr{B}}(2)$, a free $\mathbb{Z}_{\mathscr{B}}$-module of rank $23$. The algebraic fundamental group $\pi^{alg}_{1}(S, \overline{b})$ acts on $H^{4}(\overline{b})$. 
      \\\indent Consider the isometry group ${\rm Isometry}(L_{0, \mathbb{Z}_{\mathscr{B}}}, H^{4}(\overline{b}))$ and its subgroup 
     \begin{align}
         \widetilde{{\rm Isometry}}(L_{0, \mathbb{Z}_{\mathscr{B}}}, H^{4}(\overline{b})):=\{\alpha\in {\rm Isometry}(L_{0, \mathbb{Z}_{\mathscr{B}}}, H^{4}(\overline{b}))~|~\alpha(v)=c_{1}(\lambda_{\overline{b}})^{2}\}. \notag 
     \end{align}
     The group $K_{\mathscr{B}}$ acts on $L_{0, \mathbb{Z}_{\mathscr{B}}}$,  hence on $\widetilde{{\rm Isometry}}(L_{0, \mathbb{Z}_{\mathscr{B}}}, H^{4}(\overline{b}))$. The monodoromy action $\pi^{alg}_{1}(S, \overline{b})\to {\rm O}(H^{4}(\overline{b}))$ then induces an action on the quotient 
     \begin{align}
         K_{N, \mathscr{B}}\backslash\widetilde{{\rm Isometry}}(L_{0, \mathbb{Z}_{\mathscr{B}}}, H^{4}(\overline{b})).\notag
     \end{align}
     We therefore obtain the set 
     \begin{align}\label{level}
        \Big\{K_{N, \mathscr{B}}\backslash \widetilde{{\rm Isometry}}(L_{0, \mathbb{Z}_{\mathscr{B}}}, H^{4}(\overline{b}))\Big\}^{\pi_{1}^{alg}(S, \overline{b})}. 
     \end{align}
     \begin{prop}
         Let $S$ be a connected scheme over $\mathbb{Z}[1/N]$ and  $\overline{b} : {\rm Spec}(k(\overline{b}))\to S$ a geometric point. For every object $\mathcal{X}=(\pi : X\to S, \lambda)$ of $\mathcal{C}_{\mathbb{Z}[1/N]}$,  there is a canonical bijection between the set of level-$N$ structures on $\mathcal{X}$ (in $\widetilde{\mathcal{C}^{[N]}}$) and the set (\ref{level}).
     \end{prop}
     \begin{proof}
         Given an element $\alpha K_{N, \mathscr{B}}$ in (\ref{level}), reducing  modulo $N$  yields a $\pi^{alg}_{1}(S, \overline{b})$-invariant isometry $\alpha:L_{0, \mathbb{Z}/N\mathbb{Z}}\to \overline{b}^{\ast}R^{4}_{et}\pi_{\ast}\mathbb{Z}/N\mathbb{Z}(2)$, independent of the choice of representative of  $\alpha K_{N, \mathscr{B}}$. By \cite[Lemma 59.65.1 (1)]{stk}, this extends uniquely to an isometry of \'etale sheaves $(L_{0, \mathbb{Z}/N\mathbb{Z}})_{S}\to R^{4}_{et}\pi_{\ast}\mathbb{Z}/N\mathbb{Z}(2)$ inducing $\alpha$.
         \\\indent Conversely, any such isometry  $(L_{0, \mathbb{Z}/N\mathbb{Z}})_{S}\to R^{4}_{et}\pi_{\ast}\mathbb{Z}/N\mathbb{Z}(2)$ induces an isometry $L_{0, \mathbb{Z}/N\mathbb{Z}}\to \overline{b}^{\ast}R^{4}_{et}\pi_{\ast}\mathbb{Z}/N\mathbb{Z}(2)$, which is $\pi^{alg}_{1}(X, \overline{b})$-invariant (\cite[Lemma 59.65.1 (1)]{stk}). By strong approximation, there exists a marking $\alpha : L_{0, \mathbb{Z}_{\mathscr{B}}}\to H^{4}(\overline{b})$ inducing the isometry $L_{0, \mathbb{Z}/N\mathbb{Z}}\to \overline{b}^{\ast}R^{4}_{et}\pi_{\ast}\mathbb{Z}/N\mathbb{Z}(2)$. Then the orbit $\alpha K_{N, \mathscr{B}}$ is $\pi^{alg}_{1}(S, \overline{b})$-invariant.
     \end{proof}
   \begin{rmk}
       
    If  $\overline{b'}$ is another geometric point of $S$, there is a canonical bijection between the set (\ref{level}) and 
    \begin{align}\label{b^'}
         \Big\{ K_{N, \mathscr{B}}\backslash \widetilde{{\rm Isometry}}(L_{0, \mathbb{Z}_{\mathscr{B}}}, H^{4}(\overline{b^{'}}))\Big\}^{\pi^{alg}_{1}(S, \overline{b^{'}})}.
    \end{align}
   Indeed, let $h:=c_{1}(\lambda_{\overline{b}})$ and $h^{'}:=c_{1}(\lambda_{\overline{b^{'}}})$. We can choose an isomorphism
    \begin{align}\label{fun}
        \delta_{\pi} : \pi^{alg}_{1}(S, \overline{b})\to \pi^{alg}_{1}(S, \overline{b^{'}})
    \end{align}
    and a unique isometry
    \begin{align}
     \delta_{et}:   H^{4}_{et}(X_{\overline{b}}, \mathbb{Z}_{\mathscr{B}}(2))\to H^{4}_{et}(X_{\overline{b^{'}}}, \mathbb{Z}_{\mathscr{B}}(2))\notag 
    \end{align}
    mapping $h^{2}$ to $h^{'2}$,  compatible with $\delta_{\pi}$. For $\alpha$ in (\ref{level}), set $\delta_{et}\circ \alpha$ in $K_{N, \mathscr{B}}\backslash \widetilde{{\rm Isometry}}(L_{0, \mathbb{Z}_{\mathscr{B}}}, H^{4}(\overline{b^{'}}))$. For any $\gamma\in \pi^{alg}_{1}(S, \overline{b^{'}})$, 
    \begin{align}
        \gamma(\delta_{et}\circ \alpha)=(\gamma\delta_{et})\circ \alpha=\delta_{et}\circ(\delta_{\pi}^{-1}(\gamma)\alpha)=\delta_{et}\circ \alpha, \notag 
    \end{align}
so $\delta_{et}\circ \alpha$ is $\pi^{alg}_{1}(S, \overline{b^{'}})$-invariant. The deference between two choices of $\delta_{\pi}$  is an inner automorphism of $\pi^{alg}_{1}(S, \overline{b})$ (\cite[Complement 9.5.14]{ss}), hence the construction  is independent of  choices. This gives  the canonical bijection between (\ref{level}) and (\ref{b^'}). 
\end{rmk}
    Via the functor $p : \widetilde{\mathcal{C}^{[N]}}\ni (\pi:X\to S, \lambda, \alpha)\mapsto S\in (Sch/\mathbb{Z}[1/N])$, the category $\widetilde{\mathcal{C}^{[N]}}$ is a category over $\mathbb{Z}[1/N]$.
   
     \begin{lemma}
         The category $\widetilde{\mathcal{C}^{[N]}}$ is a CFG over $\mathbb{Z}[1/N]$.
     \end{lemma}
     \begin{proof}
         We must show that the category $p : \widetilde{\mathcal{C}^{[N]}}\to (Sch/\mathbb{Z}[1/N])$ is a fibered category over $(Sch/\mathbb{Z}[1/N])$. Let $f : S_{1}\to S_{2}$ be a morphism in $(Sch/\mathbb{Z}[1/N])$ and let $(\pi : X\to S_{2}, \lambda, \alpha) \in \widetilde{\mathcal{C}^{[N]}}$ be an object with $p(\pi, \lambda, \alpha)=S_{2}$. Set  $X_{S_{1}}\coloneq X\times_{S_{2}, f}S_{1}$,  and denote by $\lambda_{S_{1}}$  and $\alpha_{S_{1}}$ the corresponding pullbacks. Then the triple $(X_{S_{1}}, \lambda_{S_{1}}, \alpha_{S_{1}})$ defines an object of $\widetilde{\mathcal{C}^{[N]}}$ with $p(X_{S_{1}}, \lambda_{S_{1}}, \alpha_{S_{1}})=S_{1}$. The canonical morphism $\varphi : (X_{S_{1}}, \lambda_{S_{1}}, \alpha_{S_{1}})\to (X, \lambda, \alpha)$ induced by $f : S_{1}\to S_{2}$ satisfies 
         $p(\varphi)=f$. It remains to verify that $\varphi$ is cartesian. That is, for any object $(Y\to S, \mu, \beta)\in \widetilde{\mathcal{C}^{[N]}}$ with a morphism $\psi : (Y, \mu, \beta)\to (X, \lambda, \alpha)$ whose projeciton $p(\psi)$ factors as 
         \begin{align}
             \notag 
             S\xrightarrow{h} S_{1}\xrightarrow{f} S_{2},
         \end{align}
          there exists a unique morphism $ (Y, \mu, \beta)\to (X_{S_{1}}, \lambda_{S_{1}}, \alpha_{S_{1}})$ such that the following diagram commutes:
     \[    \begin{tikzcd}
             (Y, \mu, \beta)\arrow[bend left]{rr}{\psi}\arrow[dashed]{r}\arrow[maps to]{d}& (X_{S_{1}}, \lambda_{S_{1}}, \alpha_{S_{1}})\arrow[maps to]{d} \arrow[r, "\varphi"]& (X, \lambda, \alpha)\arrow[maps to]{d}
             \\ S\arrow[r, "h"']\arrow[rr, bend right, "p(\psi)"']& S_{1}\arrow[r, "f"']& S_{2}.
         \end{tikzcd}
         \]
         By definition, $\psi$ consists of the morphism $p(\psi) : S\to S_{2}$ and an isomorphism $Y\xrightarrow{\sim} X_{S}$ over $S$. Since $X_{S}=X_{S_{1}}\times_{S_{1}, h}S$, the pair $(h, Y\to (X_{S_{1}})_{S})$ yields the required morphism $(Y, \mu, \beta)\to (X_{S_{1}}, \lambda_{S_{1}}, \alpha_{S_{1}})$. Hence  $\widetilde{\mathcal{C}^{[N]}}$ is a fibered category over $(Sch/\mathbb{Z}[1/N])$.
         \\\indent Finally,  fix a scheme $S$ over $\mathbb{Z}[1/N]$. Any morphism in the fiber category $\widetilde{\mathcal{C}^{[N]}}(S)$ is  of the form $(id_{S}, X_{1}\cong_{S} X_{2})$; thus $\widetilde{\mathcal{C}^{[N]}}$ is a groupoid over $\mathbb{Z}[1/N]$.
     \end{proof}
     \begin{lemma}
         The CFG $\widetilde{\mathcal{C}^{[N]}}$ is a stack over $\mathbb{Z}[1/N]$ for the \'etale topology; that is,  the following conditions hold:
         \\(1) For any $S\in (Sch/\mathbb{Z}[1/N])$ and any two objects $\mathcal{X}, \mathcal{Y}\in \widetilde{\mathcal{C}^{[N]}}(S)$, the presheaf 
         \begin{align}
             {\rm Isom}_{S}(\mathcal{X}, \mathcal{Y}) : (Sch/S)&\to Sets \notag
             \\(\pi : S^{'}\to S)&\mapsto {\rm Hom}_{\widetilde{\mathcal{C}^{[N]}}(S^{'})}(\pi^{\ast}\mathcal{X}, \pi^{\ast}\mathcal{Y}) \notag 
             \end{align}
         is a sheaf for the \'etale topology.
         \\(2) Let $S\in(Sch/\mathbb{Z}[1/N])$  and  $(S_{i})_{i\in I}$ be an  \'etale covering of $S$. Suppose we are given objects $\mathcal{X}_{i}\in \widetilde{\mathcal{C}^{[N]}}(S_{i})$ and isomorphisms of $\widetilde{\mathcal{C}^{[N]}}(S_{ij})$
         \begin{align}
             \varphi_{ij} : \mathcal{X}_{j}|_{S_{ij}}\to \mathcal{X}_{i}|_{S_{ij}}, ~\varphi_{ij}|_{S_{ijk}}\circ\varphi_{jk}|_{S_{ijk}}=\varphi_{ik}|_{S_{ijk}}, \notag 
         \end{align}
          where $S_{ij}=S_{i}\times_{S}S_{j}$ and $S_{ijk}=S_{i}\times_{S}S_{j}\times_{S}S_{k}$. 
         Then there exist an object $\mathcal{X}\in \widetilde{\mathcal{C}^{[N]}}(S)$ and isomorphism of $\widetilde{\mathcal{C}^{[N]}}(S_{i})$
             $\psi_{i}: \mathcal{X}|_{S_{i}}\to \mathcal{X}_{i}$ 
        such that 
        \begin{align}\label{desce}\varphi_{ij}=\psi_{i}|_{S_{ij}}\circ (\psi_{j}|_{S_{ij}})^{-1}.
        \end{align}
       
     \end{lemma}
     \begin{proof}(1) The functor (\ref{Fd}) induces an isomorphism of \'etale presheaves 
     \begin{align}\label{F}
                 {\rm Isom}_{S}(\mathcal{X}, \mathcal{Y}) \to {\rm Isom}_{S}(F(\mathcal{X}), F(\mathcal{Y})).
            \end{align}
            Since ${\rm Isom}_{S}(F(\mathcal{X}), F(\mathcal{Y}))$ is a sheaf, so is ${\rm Isom}_{S}(\mathcal{X}, \mathcal{Y})$.
    \\ (2) Write $\mathcal{X}_{i}=(\pi_{i}:X\to S_{i}, \lambda_{i}, \alpha_{i})$. Because  $\mathcal{C}_{\mathbb{Z}[1/N]}$ is a stack, the given data yield an object $(\pi : X\to S , \lambda)\in \mathcal{C}_{\mathbb{Z}[1/N]}$ and isomorphisms $(f_{S_{i}}, f_{i}) : (\pi :X\to S, \lambda)|_{S_{i}}\to (\pi_{i} : X\to S_{i}, \lambda_{i})$ in $\mathcal{C}_{\mathbb{Z}[1/N]}$ satisfying (\ref{desce}). For each $i \in I$, consider the isomorphism of \'etale sheaves:
    \begin{align}
        R^{4}_{et}\pi_{\ast}\mathbb{Z}/N\mathbb{Z}(2)|_{S_{i}}\xrightarrow{f^{\ast}_{i}} R^{4}_{et}\pi_{i\ast}\mathbb{Z}/N\mathbb{Z}(2)\xrightarrow{\alpha_{i}} (L_{0, \mathbb{Z}/N\mathbb{Z}})_{S}. \notag 
    \end{align}
         The family $(f^{\ast}_{i}\circ \alpha_{i})_{i}$  satisfies the gluing condition for the \'etale covering $(S_{i})_{i\in I}$ of $S$. Thus we  obtain a level-$N$ structure $\alpha$ on  $(\pi : X\to S, \lambda)$, and the triple $(\pi :X\to S, \lambda, \alpha)$ satisfies (\ref{desce}).
     \end{proof}

        \subsection{The proof of Theorem \ref{main1}}

        \begin{lemma}\label{dm}For every integer $N\geq 1$, 
            the $\mathbb{Z}[1/N]$-stack $\widetilde{\mathcal{C}^{[N]}}$ is a Deligne-Mumford stack locally of finite type. If $N\geq 3$ and $N$ is coprime to $2310$, then $\widetilde{\mathcal{C}^{[N]}}$ is an algebraic space over $\mathbb{Z}[1/N]$.
        \end{lemma}
        \begin{proof}
            We verify that $\widetilde{\mathcal{C}^{[N]}}$ satisfies  conditions (i)-(iv) of \cite[Corollary 10.11]{lau}. 
            \\\textbf{Step 1}. For any $\mathbb{Z}[1/N]$-scheme $S$ and objects $\mathcal{X}, \mathcal{Y}\in \widetilde{\mathcal{C}^{[N]}}$, the functor 
            \begin{align}
                {\rm Isom}_{S}(\mathcal{X}, \mathcal{Y}) : (Sch/S)\to Sets \notag 
            \end{align}is a separated, unramified algebraic space of finite type over $S$. 
            \\ \indent Since $\mathcal{C}^{[N]}$ is algebraic and locally of finite type,  ${\rm Isom}_{S}(F(\mathcal{X}), F(\mathcal{Y}))$ is a separated algebraic space of finite type over $S$ by \cite[Corollary 10.11]{lau}. Hence, by (\ref{F}), ${\rm Isom}_{S}(\mathcal{X}, \mathcal{Y})$ has the same properties. 
            \\\textbf{Step 2}. The $\mathbb{Z}[1/N]$-stack $\widetilde{\mathcal{C}^{[N]}}$ is locally of finite presentation.
            \\\indent The natural morphism $p : \widetilde{\mathcal{C}^{[N]}}\to \mathcal{C}$ is limit-preserving on objects; therefore $p$ is locally of finite presentation. Since $\mathcal{C}$ itself is locally of finite presentation, so is $\widetilde{\mathcal{C}^{[N]}}$.
            \\\textbf{Step 3}. Let $k$ be a field of finite type over $\mathbb{Z}[1/N]$ and $\overline{x}\in \widetilde{\mathcal{C}^{[N]}}({\rm Spec}(k))$. There exist a finite separable extension $k^{'}/k$, a complete local noetherian $\mathcal{O}_{S}$-ring $A^{'}$ with  residue field $k^{'}$, and a point $x^{'}\in \widetilde{\mathcal{C}^{[N]}}({\rm Spec}(A^{'}))$ such that: 
            \\(a) the diagram of $\mathbb{Z}[1/N]$-stacks 
            \[\xymatrix{{\rm Spec}(k^{'})\ar[r]\ar[d]& {\rm Spec}(k)\ar[d]^{\overline{x}}\\ {\rm Spec}(A^{'})\ar[r]_{x^{'}}& \widetilde{\mathcal{C}^{[N]}}}\]
            is 2-commutative; and
            \\(b) for every local artinian $\mathcal{O}_{S}$-ring $B^{'}$ with  residue field $k^{'}$, every square-zero ideal $I^{'}\subset B^{'}$, any local homomorphism $A^{'}\to B^{'}/I^{'}$ of $\mathcal{O}_{S}$-rings, and any $y^{'}\in\widetilde{\mathcal{C}^{[N]}}({\rm Spec}(B^{'}))$ making the corresponding diagram 
            \[\xymatrix{{\rm Spec}(B^{'}/I^{'})\ar[r]\ar[d]& {\rm Spec}(A^{'})\ar[d]^{x^{'}}\\ {\rm Spec}(B^{'})\ar[r]_{y^{'}}&\widetilde{\mathcal{C}^{[N]}}}\]
             2-commutative, there exists a lifting $A^{'}\to B^{'}$ of $A^{'}\to B^{'}/I$ which makes the following diagram 
            \[\xymatrix{{\rm Spec}(B^{'}/I^{'})\ar[r]\ar[d]& {\rm Spec}(A^{'})\ar[d]^{x^{'}}\\ {\rm Spec}(B^{'})\ar[ru]\ar[r]_{y^{'}}&\widetilde{\mathcal{C}^{[N]}}}\]2-commutative.
            \\$\indent$ Indeed, because  $\mathcal{C}$ is an algebraic $\mathbb{Z}[1/N]$-stack locally of finite type, we can find   $k^{'}, A^{'}$ satisfying (a) and (b) for $k$ and $p\circ \overline{x}$ by \cite[Corollary 10.11]{lau}. Write $\overline{x}=(X\to k, \alpha)\in\widetilde{\mathcal{C}^{[N]}}({\rm Spec}(k))$ and let $x^{'}=(X^{'}\to {\rm Spec}(A^{'}))\in \mathcal{C}({\rm Spec}(A^{'}))$. Condition (a) gives $X_{k^{'}}=X^{'}_{k^{'}}$.  The pair  $(X^{'}_{k^{'}}\to X^{'}, {\rm Spec}(k^{'})\to {\rm Spec}(A^{'}))$ yields a morphism $X^{'}_{k^{'}}\to X^{'}$ in $\mathcal{C}$. Using this morphism, we can introduce the level structure $\alpha$ on $X_{k^{'}}$ to a level stucture on $X^{'}$. The pair $(X^{'}\to {\rm Spec}(A^{'}), \alpha)$ defines a point $x^{''}\in \widetilde{\mathcal{C}^{[N]}}({\rm Spec}(A^{'}))$, and by construction we obtain the commutative diagram as in (a).
           Statement (b) is proved similarly.
            \\\textbf{Step 4}. Let $U$ be an affine $S$-scheme  of finite type,  $u\in U$ a closed point, and $x\in \widetilde{\mathcal{C}^{[N]}}(U)$  formally smooth at $u$. Then there exists an affine open set $ V\subset U$ containing $u$ such that the induced morphism $V\to \widetilde{\mathcal{C}^{[N]}}$ is smooth.
            \\$\indent$Applying  \cite[Corollary 10.11]{lau} to $F\circ x\in \mathcal{C}^{[N]}(U)$ yields such a $ V $ with  $V\to \mathcal{C}^{[N]}$ smooth. For any morphism $f : T\to \widetilde{\mathcal{C}^{[N]}}$ from a scheme $T$, there is a canonical isomorphism of algebraic spaces $V\times_{f, \widetilde{\mathcal{C}^{[N]}}}T\to V\times_{F\circ f, \mathcal{C}^{[N]}}T,$ hence  $V\to \widetilde{\mathcal{C}^{[N]}}$ is also smooth.
            \\\indent\textbf{The proof of Lemma \ref{dm}}. By steps 1-4 and  \cite[Corollary 10.11]{lau}, $\mathbb{Z}[1/N]$-stack $\widetilde{\mathcal{C}^{[N]}}$ is an algebraic stack locally of finite type. Its diagonal  $\widetilde{\mathcal{C}^{[N]}}\to \widetilde{\mathcal{C}^{[N]}}\times_{\mathbb{Z}[1/N]}\widetilde{\mathcal{C}^{[N]}}$ is unramified; hence  $\widetilde{\mathcal{C}^{[N]}}$ is a Deligne-Mumford stack  (\cite[Theorem 8.1]{lau}).
            \\\indent If $N\geq 3$ is coprime to $2310$, $\mathcal{C}^{[N]}$ is  algebraic (\cite[Theorem 3.2, (3)]{jal}). By \cite[Corollary 8.3.5]{ols}, each object of  $\mathcal{C}^{[N]}(U)$ has trivial automorphism group for every $\mathbb{Z}[1/N]$-scheme $U$; the isomorphism (\ref{F}) shows the same for $\widetilde{\mathcal{C}^{[N]}}(U)$. Applying   \cite[Corollary 8.3.5]{ols} again, we conclude that  $\widetilde{\mathcal{C}^{[N]}}$ is an algebraic space.
        \end{proof}
        \textbf{Proof of Theorem \ref{main1}.}
        \\\indent Note first that the morphism $F : \widetilde{\mathcal{C}^{[N]}}\to \mathcal{C}^{[N]}$ of algebraic stacks is representable. We prove that  $F$ is \'etale. Let $T\to \mathcal{C}^{[N]}$ be a morphism from a scheme $T$, and choose an \'etale surjection $U\to \widetilde{\mathcal{C}^{[N]}}$ from a scheme $U$. Consider the  commutative diagram 
        \[\xymatrix{T\times_{\mathcal{C}^{[N]}}\widetilde{\mathcal{C}^{[N]}}\ar[r]& T\\  T\times_{\mathcal{C}^{[N]}}U\ar[r]\ar[u]& T\times_{\widetilde{\mathcal{C}^{[N]}}}U.\ar[u]}
        \]
        The right vertical morphism is \'etale (since $U\to \widetilde{\mathcal{C}^{[N]}}$ is \'etale) and the bottom horizontal morphism is an isomorphism. Hence the composite $T\times_{\mathcal{C}^{[N]}}U\to T\times_{\widetilde{\mathcal{C}^{[N]}}}U\to T$ is  \'etale. By  definition of $U$, the left vertical arrow is \'etale and surjective;  therfore the top horizontal morphism is \'etale. This shows that $F$ is \'etale (\cite[Lemma 35.14.5]{stk}). 
        \\\indent Consequently, the $\mathbb{Z}[1/N]$-stack $\widetilde{\mathcal{C}^{[N]}}$ is a scheme (\cite[II, Corollary 6.17]{knu}). Moreover, $F$ is an open immersion because it is a monomorphism (\cite[Theorem 41.14.1]{stk}); thus the scheme $\widetilde{\mathcal{C}^{[N]}}$ is affine. 
        
        \section{The arithmetic period Map of Cubic Fourfolds; the Proof of Theorem \ref{main2} and Theorem\ref{rank21ex}}\label{period main2}
     
\subsection{Definitions of the period map}\label{defperiod}Let $N\geq 3$ be an integer  coprime to 2310. By Theorem \ref{main1}, the stack $\widetilde{\mathcal{C}^{[N]}}$ is represented by a smooth affine scheme of finite type over $\mathbb{Z}[1/N]$. 
\\\indent In this section, we  define the period map:
\begin{align}\label{pdef}
    j_{N, \mathbb{C}} : \widetilde{\mathcal{C}^{[N]}}_{\mathbb{C}}\to {\rm Sh}_{K_N}(L)_{\mathbb{C}}
\end{align}Take $(X, h, \alpha)\in \widetilde{\mathcal{C}^{[N]}}(\mathbb{C})$. Let $\widetilde{\alpha} : L_{0, \mathbb{Z}_{\mathscr{B}}}\to H^{4}_{et}(X, \mathbb{Z}_{\mathscr{B}}(2))$ be a representative of the class $\alpha$. Choose an isometry $a : H^{4}(X, \mathbb{Z}(2))\to L_{0}$ such that $a(h^{2})=v$ and $a\circ \widetilde \alpha \in{\rm SO}(V)(\hat{\mathbb{Z}})$. Let $h_{X} : \mathbb{S}\to {\rm SO}(P^{4}(X, \mathbb{R}(2)))$ be the polarized Hodge structure of $(X, h^{2})$. 
\\\indent  We claim that the correspondence 
\begin{align}\label{const j}\widetilde{\mathcal{C}^{[N]}}(\mathbb{C})\ni (X, h, \alpha)\mapsto [a\circ h_{X}\circ a^{-1}, a\circ \widetilde{\alpha}]\in{\rm Sh}_{K_{N}}(L)(\mathbb{C})
\end{align}
is independent of the choices of $a$ and $\widetilde{\alpha}$. 
\\\indent Indeed, suppose  $a^{'} : H^{4}(X, \mathbb{Z}(2))\to L_{0}$ is another isometry with $a^{'}(h^{2})=v$ and $a^{'}\circ (\widetilde{\alpha}\circ\kappa) \in {\rm SO}(V)(\hat{\mathbb{Z}})$ for some $\kappa\in K_{N}$. Then we can find  $g\in {\rm SO}(V)(\mathbb{Z})$ with $a^{'}=g\circ a$, since $g\circ (a\circ \widetilde{\alpha})\circ \kappa\in {\rm SO}(V)(\hat{\mathbb{Z}})$. Hence
\begin{align}
    [(a^{'}\circ h_{X}\circ a^{'-1}), a^{'}\circ (\widetilde{\alpha}\circ \kappa)]
   =&[g\cdot(a\circ h_{X}\circ a^{-1}, (a\circ\widetilde{\alpha})\circ \kappa)]\notag 
   \\=& [(a\circ h_{X}\circ a^{-1}, a\circ \widetilde{\alpha})].\notag 
    \end{align} Thus (\ref{const j}) is well defined;  we denote it by $j_{N, \mathbb{C}}$.
    \begin{rmk}\label{mod1}
        Under the correspondence (\ref{const j}), if $a^{'} : H^{4}(X, \mathbb{Q}(2))\to V_{0}$ is an isometry with $a^{'}(h^{2})=v$ and $a^{'}\circ \widetilde{\alpha}\in {\rm SO}(V)(\mathbb{A}_{f})$, then  $[(a\circ h_{X}\circ a^{-1}, a\circ \widetilde{\alpha})]=[(a^{'}\circ h_{X}\circ a^{'-1}, a^{'}\circ \widetilde{\alpha})]$. In fact, 
        \begin{align}
            a^{'}\circ a^{-1}=(a^{'}\circ \widetilde{\alpha})\circ (\widetilde{\alpha}^{-1}\circ a^{-1})\in {\rm O}(V)(\mathbb{Q})\cap {\rm SO}(V)(\mathbb{A}_{f})=G(\mathbb{Q}), \notag 
            \end{align}and therefore 
        \begin{align}
            [(a\circ h_{X}\circ a^{-1}, a\circ \widetilde{\alpha})]&=[(a\circ a^{'-1})\circ a^{'}\circ h_{X}\circ a^{'-1}\circ (a\circ a^{'-1}), (a\circ a^{'-1})\circ a^{'}\circ \widetilde{\alpha})] \notag 
            \\& =[(a^{'}\circ h_{X}\circ a^{'-1}, a^{'}\circ \widetilde{\alpha})]. \notag 
        \end{align}
    \end{rmk}
      \begin{lemma}
           The map $j_{N, \mathbb{C}} : \widetilde{\mathcal{C}^{[N]}}(\mathbb{C})\to {\rm Sh}_{K_{N}}(L)(\mathbb{C})$ is an algebraic \'etale morphism. Hence $j_{N, \mathbb{C}}$ induces a unique \'etale morphism (\ref{pdef}) of $\mathbb{C}$-schemes.
      \end{lemma}
      \begin{proof}
           Let $(X\to \widetilde{\mathcal{C}^{[N]}}_{\mathbb{C}}, \lambda, \alpha)$ be the universal object of  $\widetilde{\mathcal{C}^{[N]}}_{\mathbb{C}}$, and let $V$ be a connected component of $\widetilde{\mathcal{C}^{[N]}}_{\mathbb{C}}$. Write $(\pi : X_{V}\to V, \lambda_{V}, \alpha_{V})$ for the pullback to $V$. Since $\widetilde{\mathcal{C}^{[N]}}_{\mathbb{C}}$ is noetherian, each $V$ is a Zariski open set in $\widetilde{\mathcal{C}^{[N]}}_{\mathbb{C}}$. We show that the map of analytic spaces
        \begin{align}
         j_{N, \mathbb{C}}|_{V^{an}} : V^{an}\to {\rm Sh}_{K_{N}}(L)(\mathbb{C}), s\mapsto j_{N, \mathbb{C}}(X_{s}, \lambda_{s}, \alpha_{s})
        \end{align}
        is holomorphic and  a locla  isomorphism. Here $\alpha_{s}\in K_{N}\backslash{\rm Isometry}(L_{0, \mathbb{Z}_{\mathscr{B}}}, H^{4}_{et}(X_{s}, \mathbb{Z}_{\mathscr{B}}(2)))$ is the level-$N$ structure $\alpha_{V}$ corresponding to the geometric point $s$.
        \\\indent For each $s\in V^{an}$, choose a simply connected neighborhood $V_{s}$ on which  the local system $R^{4}\pi_{\ast}\mathbb{Z}(2)|_{V_{s}}$ is constant. Fix a marking $R^{4}\pi_{\ast}\mathbb{Z}(2)|_{V_{s}}\cong L_{0, V_{s}}$ sending $\lambda_{V_{s}}\mapsto v$; then the VHS $R^{4}\pi_{\ast}\mathbb{Z}(2)|_{V_{s}}$ is polarizable. By Griffiths \cite[Proposition 9.3]{gri}, there is a (holomorphic) local period map 
        \begin{align}
j :       V_{s}\to \Omega^{+}    \notag   
        \end{align}
        given by  $s_{1}\mapsto a\circ h_{X_{s_{1}}}\circ a^{-1}$, where $h_{X_{s_{1}}} : \mathbb{S}\to {\rm SO}(P^{4}(X_{V_{s_{1}}}, \mathbb{R}(2)))$ is the polarized Hodge structure of $(X_{V_{s_{1}}}, \lambda_{V_{s_{1}}})$ and $a : H^{4}(X_{V_{s_{1}}}, \mathbb{Z}(2))\to L_{0}$ is the stalk at $s_{1}$ of the chosen marking. From the proof of Theorem \ref{main1} we know that $F : \widetilde{\mathcal{C}^{[N]}}\to \mathcal{C}^{[N]}$ is an open immersion; the forgetful morphism $\mathcal{C}^{[N]}\to \mathcal{C}$ is \'etale (\cite[the proof of Theorem 3.2]{jal}). Thus the natural morphism $p : \widetilde{\mathcal{C}^{[N]}}\to \mathcal{C}$ is \'etale, and the induced holomorphic map $p^{an} : \widetilde{\mathcal{C}^{[N]}}(\mathbb{C})\to \mathcal{C}(\mathbb{C})$ is a local isomorphism.  Consider the commutative diagram 
        \[\xymatrix{\mathcal{C}(\mathbb{C})\ar[r]&{\rm O}(L)\backslash\Omega^{+}\\ V_{s}\ar[u]^{p^{an}|_{V_{s}}}\ar[r]_{j}&\Omega^{+}\ar[u]}\]
        in which  the top horizontal morphism $\mathcal{C}(\mathbb{C})\to {\rm O}(L)\backslash\Omega^{+}$ is the global period map, an open immersion by Voisin's Torelli thorem for cubic fourfolds \cite{voi}. Hence  $j$ is a local isomorphism. 
       \\\indent Recall  the decomposition  (\ref{shdec}) of ${\rm Sh}_{K_{N}}(L)_{\mathbb{C}}$ into  connected components:  ${\rm Sh}_{K_{N}}(L)_{\mathbb{C}}=\coprod_{g\in \mathscr{C}} \Gamma_{g}\backslash \Omega^{+}$, where $\mathscr{C}=G(\mathbb{Q})_{+}\backslash G(\mathbb{A}_{f})/K$ and $\Gamma_{g}=G(\mathbb{Q})\cap gKg^{-1}$. To see that $j_{N, \mathbb{C}}|_{V_{s}}$ is holomorphic and a local isomorphism, it suffices to show that there exists $g\in \mathscr{C}$ such that $j_{N, \mathbb{C}}|_{V_{s}}$ is the composition of $j : V_{s}\to \Omega$ and $\Omega^{+}\to \Gamma_{g}\backslash\Omega^{+}$ .
        \\\indent By (\ref{const j}), let $\widetilde{\alpha}$ represent $\alpha_{s}$ and choose $a : H^{4}(X_{s}, \mathbb{Z}(2))\to L_{0}$ with $\alpha(\lambda_{s})=v$. Take $s_{1}\in V_{s}$. Since $V^{an}$ is path connected, there exist an isomorphism of fundamental groups
        \begin{align}
           \delta_{\pi}: \pi_{1}(V^{an}, s_{1})\to \pi_{1}(V^{an}, s)\notag 
        \end{align}
        and an isometry
        \begin{align}
            \delta_{B} : H^{4}(X_{s_{1}}, \mathbb{Z}(2))\to H^{4}(X_{s}, \mathbb{Z}(2)),\notag 
        \end{align}
        compatible with $\delta_{\pi}$  sending $\lambda_{s_{1}}$ to $\lambda_{s}$. Then the isometry
        \begin{align}
          \delta_{B}^{-1}\circ \widetilde{\alpha} : L_{0} \to H^{4}(X_{s_{1}}, \mathbb{Z}(2))\notag   
        \end{align}
        represents $\alpha_{s_{1}}$, and \begin{align}a\circ \delta_{B} : H^{4}(X_{s_{1}}, \mathbb{Z}(2))\to L_{0}. \notag 
        \end{align}
         is a marking. 
        Since $(a\circ\delta_{B})\circ (\delta^{-1}_{B}\circ \widetilde{\alpha})=a\circ \widetilde{\alpha}\in {\rm SO}(V)(\hat{\mathbb{Z}})$, we have $j_{N, \mathbb{C}}(s_{1})=[((a\circ \delta_{B})\circ h_{X_{s_{1}}}\circ (a\circ \delta_{B})^{-1}, a\circ \widetilde{\alpha})]$. Setting $g:=a\circ \widetilde{\alpha}\in \mathscr{C}$, we get $j_{N, \mathbb{C}}(s_{1})\in \Gamma_{g}\backslash\Omega^{+}$. Hence $j_{N, \mathbb{C}}|_{V_{s}}$ is  the composition of $j$ and $\Omega^{+}\to \Gamma_{g}\backslash \Omega^{+}$.
        \\\indent By \cite[Theorem 3.10]{bor}, the holomorphic map $j_{N, \mathbb{C}}|_{V^{an}}$ is algebraic. Since $j_{N, \mathbb{C}}|_{V^{an}}$ is a local isomorphism, the algebraic morphism $j_{N, \mathbb{C}}|_{V}$ is \'etale. Therefore $j_{N, \mathbb{C}}$ is  \'etale.
        \end{proof}
        \textbf{Proof of Theorem \ref{main2} :} It suffices to show that $j_{N, \mathbb{C}} : \widetilde{\mathcal{C}^{[N]}}(\mathbb{C})\to {\rm Sh}_{K_{N}}(L)(\mathbb{C})$ is injective. Let $(X_{i}, \lambda_{i}, \alpha_{i})\in\widetilde{\mathcal{C}^{[N]}}(\mathbb{C})$ for $i=1, 2$, and  suppose  $j_{N, \mathbb{C}}(X_{1}, \lambda_{1}, \alpha_{1})=j_{N, \mathbb{C}}(X_{2}, \lambda_{2}, \alpha_{2})$. Choose repreentatives $\widetilde{\alpha_{i}}$ of $\alpha_{i}$ and markings $a_{i} : H^{4}(X_{i}, \mathbb{Z}(2))\to L_{0}$ with $a_{i}(\lambda_{i})=v$. Then in ${\rm Sh}_{K_{N}}(L)_{\mathbb{C}}$ we have 
       \begin{align}
           [(a_{1}\circ h_{X_{1}}\circ a_{1}^{-1}, a_{1}\circ\widetilde{\alpha_{1}})]=[(a_{2}\circ h_{X_{2}}\circ a^{-1}_{2}, a_{2}\circ \widetilde{\alpha_{2}})].\notag 
       \end{align}
        Thus, for some $q\in G(\mathbb{Q}),$ 
        \begin{align}
            (a_{1}\circ h_{X_{1}}\circ a_{1}^{-1}, a_{1}\circ\widetilde{\alpha_{1}})=((q\circ a_{2})\circ h_{X_{2}}\circ (q\circ a_{2})^{-1}, q\circ a_{2}\circ \widetilde{\alpha_{2}}) \notag 
        \end{align}
       in $X_{L}\times (G(\mathbb{A}_{f})/K)$. Hence there exists $\kappa\in K_{N}$ such that 
        \begin{align}
            (a_{1}\circ h_{X_{1}}\circ a_{1}^{-1},  a_{1}\circ\widetilde{\alpha_{1}}\circ \kappa)=((q\circ a_{2})\circ h_{X_{2}}\circ (q\circ a_{2})^{-1}, q\circ a_{2}\circ \widetilde{\alpha_{2}}) \notag 
        \end{align}
       It follows that $q=(a_{1}\circ\widetilde{\alpha_{1}})\circ \kappa\circ (a_{2}\circ \widetilde{\alpha_{2}})^{-1}\in G(\mathbb{Z})=G(\hat{\mathbb{Z}})\cap G(\mathbb{Q})\subset G(\mathbb{A}_{f})$. Therefore $ h_{X_{1}}= (a^{-1}_{1}\circ q\circ a_{2})\circ h_{X_{2}}\circ (a_{1}^{-1}\circ q\circ a_{2})^{-1}$ induces a $\mathbb{Z}$-Hodge isometry
        \begin{align}
          a^{-1}_{1}\circ q\circ a_{2} :   H^{4}(X_{2}, \mathbb{Z}(2))\to H^{4}(X_{1}, \mathbb{Z}(2)) \notag 
        \end{align}
        sending $\lambda_{2}$ to $\lambda_{1}$. By Voisin's Torelli theorem \cite{voi}, the Hodge isometry $a^{-1}_{1}\circ q\circ a_{2}$ is induced by  an isomorphism of polarized varieties $f : (X_{1}, \lambda_{1})\to (X_{2}, \lambda_{2})$. Moreover, the relation $f^{\ast}\circ \widetilde{\alpha_{2}}=\widetilde{\alpha_{1}}\circ \kappa$ shows that $f$ respects the level structures, yielding  an isomorphism of $(X_{1}, \lambda_{1}, \alpha_{1})\cong(X_{2}, \lambda_{2}, \alpha_{2})$.  Since  morphisms in  $\widetilde{\mathcal{C}^{[N]}}(\mathbb{C})$ are identities, we conclude  $(X_{1}, \lambda_{1}, \alpha_{1})=(X_{2}, \lambda_{2}, \alpha_{2})$.
   \subsection{Application of Theorem \ref{main2}:  Rank-21 cubic fourfolds}
       \begin{dfn}\label{def21}
           A cubic fourfold $X/\mathbb{C}$ is said to be \textbf{rank-21 cubic fourfold} if its algebraic rank satisfies ${\rm rk}A(X)=21$.
       \end{dfn}
       \begin{exmp}
           (1) The Fermat cubic fourfold (\ref{Fermat}) is  rank-21. The automorphism in Example \ref{exac4},
           \begin{align}\label{ferauto}
               (x_0:\cdots:x_{5})\mapsto (\zeta_{3}x_{0}:\cdots : x_{5}) 
           \end{align}
           is non-symplectic (\cite[Lemma 6.2]{boi}). Hence its induced action on $H^{-1, 1}(X)$ is multiplication by $\zeta_{3}$. Consequently, the image of (\ref{dif}) is $\mathbb{Q}(\zeta_{3})$.
           \\(2) The cubic fourfold in Example \ref{exac4} (3) is also rank-21. The matrix $\begin{pmatrix}
               1& -1 \\ 1& 0
           \end{pmatrix}$ defines an orthogonal transformation of $T(X)\cong\begin{pmatrix}
               -22& 11\\ 11& -22.
           \end{pmatrix}$ The relation 
           \begin{align}
               \begin{pmatrix}
                   1&-1\\ 1&0
               \end{pmatrix}^{3}=-\begin{pmatrix}
                   1&0\\ 0& 1
               \end{pmatrix} \notag 
           \end{align}
        shows that the image of (\ref{dif}) is generated by a root of  $x^{2}-x+1=0$, i.e. by $\zeta_{3}$.
       \end{exmp}
       Rank-21 cubic fourfolds correspond to $singular~ K3~ surfaces$, i.e. complex K3 surfaces with Picard rank $\rho(S)=20$. According to Shioda-Inose \cite{si}, assigning to a singular K3 surface $S$ its transcendental lattice $T(S)$ yields a bijection 
       \begin{align}\label{shioda}
           \{{\rm singular~K3~surfaces}\}\leftrightarrow \{ T~|~{\rm positive~ definite,~even,~oriented~lattice~with~rank ~}2\},
       \end{align}
       where both sides are taken up to isomorphism.
       \\\indent On the other hand, for a rank-21 cubic fourfold $X$, the lattice $T(X)$ is contained in the even lattice $\langle h^{2}\rangle^{\perp}\subset H^{4}(X, \mathbb{Z}(2))$. Thus, for rank-21 cubic fourfolds, the correspondence $X\mapsto T(X)$ defines a map
       \begin{align}\label{cla}
           \{{\rm rank}~21~{\rm cubic~fourfolds}\}\to \{T~|~{\rm positive~definite,~even,~oriented~lattice~with~rank ~}2\}.
       \end{align}
        Compositing (\ref{shioda}) with (\ref{cla}), we obtain
       \begin{align}
            \{{\rm rank}~21~{\rm cubic~fourfolds}\}\to \{{\rm singular~K3~surfaces}\},
       \end{align}
       In other words, for each rank-21 cubic fourfold $X$, there exists a singular K3 surface $S$ such that $T(X)$ and $T(S)$ are isometric as oriented lattices. Since the period domain $X_{T}\subset \mathbb{P}(T_{\mathbb{C}})$ of $T$ consists of only two points,  this isometry is a Hodge isometry.  Hence, we obtain the following statement (2).
        \begin{prop}\label{21fun}
           Let $X$ be a rank-21 cubic fourfold. 
           \\(1) $X$ is of CM type, and the reflex field is an imaginary quadratic field. 
           \\(2) There exists a singular K3 surface $S$ such that we have a Hodge isometry
           \begin{align}
                H^{4}(X, \mathbb{Z}(2))\supset T(X)\cong T(S)\subset H^{2}(S, \mathbb{Z}(1)). \notag 
           \end{align}
       \end{prop}
       \begin{proof}
           It suffices to prove (1). By (2), the reflex field of $X$ equals that of $S$.  Singular K3 surfaces are always of CM type  \cite[Chapter 3, Remark 3.10]{Huy}. Thus  $K(T(S))$ is an imaginary quadratic field over $\mathbb{Q}$ by \cite[Remark 1.5.3 and Theorem 2.3.1]{Zarhin}.  Valloni \cite[Proposition 2.10]{val} shows that the reflex field is induced from the ring of Hodge endomorphisms via $K(T(X))\hookrightarrow \mathbb{C}$; hence (1) follows.
        
      \end{proof}
       An application of Theorem \ref{main2} is to give the converse of Proposition \ref{21fun} (1):
    \begin{thm}\label{apl1} For any imaginary quadratic field $E\subset \mathbb{C}$ over $\mathbb{Q}$, there exists a rank-21 cubic fourfold whose reflex field is $E\subset \mathbb{C}$. Moreover, for $N\geq 3$   coprime to 2310, the inverse image of the set $\{$rank-21 cubic fourfolds with the reflex field $E\subset \mathbb{C}$ $\}$ under the surjection $\widetilde{\mathcal{C}^{[N]}}_{\mathbb{C}}\to \mathcal{C}_{\mathbb{C}}$ is Zariski-dense  in the $\mathbb{C}$-scheme $\widetilde{\mathcal{C}^{[N]}}_{\mathbb{C}}$.
    \end{thm}
    To prove this, we recall some  lattice-theoretic lemmas.
   \\\\ \indent Fix  $E=\mathbb{Q}(\sqrt{-d})$ with $d>0$  square-free, and let $\mathcal{O}_{E}$ be its ring of integers. The $\mathbb{Z}$-module $\mathcal{O}_{E}$ is  free of rank 2. An integral basis of $\mathcal{O}_{E}$ is  $(1, (1+\sqrt{-d})/2)$ if $-d\equiv 1$ (mod $4)$, and $(1, \sqrt{-d})$ otherwise (\cite[IV, \S 2, Theorem 5]{lang}). Let $z\mapsto \overline{z}$ br  complex conjugation of $E$, which  induces an automorphism of $\mathcal{O}_{E}$. Consider 
    \begin{align}
        q : \mathcal{O}_{E}\times \mathcal{O}_E\to \mathbb{Z}, (x, y)\mapsto {\rm Tr}_{E/\mathbb{Q}}( x\overline{y}) \notag 
    \end{align}
     a non-degenerate symmetric $\mathbb{Z}$-bilinear form. Thus $(\mathcal{O}_{E}, q)$ is a lattice. 
    \\\indent We compute  $(\mathcal{O}_{E}, q)$ explicitly. If $-d\equiv 1$ (mod $4)$, then the self-intersection of $(1+\sqrt{-d})/2$ is $(1+d)/2$. The intersection matrix in the basis $(1, (1+\sqrt{-d})/2)$ is  $\begin{pmatrix}
        2& 0\\ 0& (1+d)/2
    \end{pmatrix}$. If $-d\not\equiv 1$ (mod $4)$, then $q(\sqrt{-d}, \sqrt{-d})=2d$, so the matrix in the basis $(1, \sqrt{-d})$ is
    $\begin{pmatrix}
        2& 0\\ 0 &2d 
    \end{pmatrix}$. In either cases, $(\mathcal{O}_{E}, q)$ is even of  signature $(2, 0)$. By \cite[Chapter 14, Proposition 1.8]{Huy}, there is a primitive embedding
    \begin{align}\label{pri}
        (\mathcal{O}_{E}, q)\hookrightarrow U^{\oplus 2}.
    \end{align}
    Compositing (\ref{pri}) with $U^{\oplus 2}\hookrightarrow L$,  we obtain:
    \begin{lemma}\label{pri2}
        There exists an embedding of lattices:
        \begin{align}
            (\mathcal{O}_{E}, q)\hookrightarrow L. \notag 
        \end{align}
    \end{lemma}
  Extending  scalars from $\mathbb{Z}$ to $\mathbb{Q}$, Lemma \ref{pri2} implies:

    \begin{lemma}\label{v1}
        Let $E/\mathbb{Q}$ be an imaginary quadratic field. Then there is a quadratic space $V_{1}$ over $\mathbb{Q}$ such that  
        \begin{align}
            (E, q)\oplus V_{1}\cong V . \notag 
        \end{align}
        as quadratic spaces over $\mathbb{Q}$.
    \end{lemma}
    \begin{lemma}\label{21} Let $N\geq 1$.
        For any imaginary quadratic field $E\subset \mathbb{C}$, there exists a CM point in ${\rm Sh}_{K_{N}}(L)(\mathbb{C})$ whose reflex field  is $E\subset \mathbb{C}$.
    \end{lemma}
    \begin{proof}
        By Lemma \ref{v1}, write $(E, q)\oplus V_{1}\cong V_{0}$. Fix an embedding  $\sigma : E\hookrightarrow \mathbb{C}$. The other embedding  is given by the composition of $\sigma$ and complex conjugation. Set $\mathbb{C}_{\sigma}:=E\otimes_{E, \sigma}\mathbb{C}$,  $\mathbb{C}_{\overline{\sigma}}:=E\otimes_{E, \overline{\sigma}}\mathbb{C}$; these are one-dimensional $\mathbb{C}$-vector spaces. There is a canonical isomorphism 
\begin{align}\label{cc}
            E\otimes_{\mathbb{Q}}\mathbb{C}\cong \mathbb{C}_{\sigma}\oplus \mathbb{C}_{\overline{\sigma}}.
        \end{align}
        given by $x\otimes c\mapsto (1\otimes \sigma(x)c, 1\otimes \overline{\sigma}(x)c)$.
       Choose a square root $\sqrt{-d}$ of $-d$ with $\sigma(\sqrt{-d})=\sqrt{-d}$. Then (\ref{cc}) identifies 
        \begin{align}
            \Big\{1\otimes\frac{c}{2} + \sqrt{-d}\otimes \frac{c}{2\sqrt{-d}}\in E\otimes_{\mathbb{Q}}\mathbb{C} ~|~ c\in\mathbb{C}\Big\}&\cong \mathbb{C}_{\sigma},  \notag 
            \\ \Big\{1\otimes\frac{c}{2} - \sqrt{-d}\otimes \frac{c}{2\sqrt{-d}}\in E\otimes_{\mathbb{Q}}\mathbb{C}~ | ~c\in\mathbb{C}\Big\}&\cong \mathbb{C}_{\overline{\sigma}}.\notag 
        \end{align}
      Define a $\mathbb{Q}$-HS of K3 type on $V$ by
       \begin{align}
           V^{p, q}:=\begin{cases}
               \mathbb{C}_{\sigma}, &(p, q)=(1, -1)
               \\ V_{1}\otimes_{\mathbb{Q}}\mathbb{C}, &(p, q)=(0, 0)
               \\ \mathbb{C}_{\overline{\sigma}}, &(p, q)=(-1, 1).
           \end{cases}\notag 
       \end{align}
       Complex conjugation on $E\otimes_{\mathbb{Q}}\mathbb{C}$ interchanges $\mathbb{C}_{\sigma}$ and $\mathbb{C}_{\overline{\sigma}}$, so the above defines a polarized $\mathbb{Q}$-HS $h\in X_{L}$: for  any nonzero  $x\in V^{1, -1}$, one has $x^{2}=0$ and $(x, \overline{x})>0$. 
       \\\indent The action of $\mathbb{S}$ on $V^{1}$ by $h$ is trivial (Example \ref{00}),  hence $h$ factors through  $\mathbb{S}\to {\rm SO}(E, q)$. In the basis $(1, \sqrt{-d})$ for $E$,  we have an isomorphism of $\mathbb{Q}$-algebraic groups ${\rm SO}_{2}(E, q)\cong \Big\{\begin{pmatrix}
           a& -db\\ b& a
       \end{pmatrix}~|~a^{2}+db^{2}=1\Big\}$, so over $E$ there is a splitting 
       \begin{align}
           {\rm SO}(E, q)_{E}\cong \mathbb{G}_{m, E}. \notag 
       \end{align}
      In particular,  $h$ is special  and its reflex field is $\sigma:E\to \mathbb{C}$.
       
    \end{proof}
    We will use the following lemma to prove Theorem \ref{apl1}.
     \begin{lemma}(\cite[Lemma 13.5]{milne}) For any $x\in X_{L}$, the set $\{[(x, a)] ~|~ a\in G(\mathbb{A}_{f})\}$ is Zariski-dense in ${\rm Sh}_{K_{N}}(L)(\mathbb{C})$.
    \end{lemma}
    \textbf{Proof of Theorem \ref{apl1} :}
        Let $\sigma : E\subset \mathbb{C}$ be an imaginary quadratic field over $\mathbb{Q}$. By Lemma \ref{21}, there exists a CM point $x\in X_{L}$ with   reflex field  $\sigma : E\to \mathbb{C}$. By Theorem \ref{main2}, the image of $j_{N, \mathbb{C}}$ is open  and  contains a point $[(x, a)]$ for some $a\in G(\mathbb{A}_{f})$. Hence there exists  a cubic fourfold $X$ of CM type with  reflex field $E\subset \mathbb{C}$. Since the transcendental rank is $2$, the algebraic rank of $X$ is 21.

    \section{The proof of Theorem \ref{main3} }\label{secmain3}
   
    \subsection{Complex multiplication for rank-21 cubic fourfolds}
     In this section  we prove the following.
    \begin{thm}(Complex multiplication for rank-21 cubic fourfolds)\label{cmt}
Let $X/\mathbb{C}$ be a rank-21 cubic fourfold, and let $E\subset \mathbb{C}$ be its reflex field. For $\tau\in{\rm Aut}(\mathbb{C}/E)$ and $s\in \mathbb{A}^{\times}_{E}$ satisfying ${\rm art}_{E}(s)=\tau_{E^{ab}}$, there exists a unique Hodge isometry $\eta : T(X)_{\mathbb{Q}}\to T(X^{\tau})_{\mathbb{Q}}$ such that the diagram commutes:
\[\xymatrix{T(X)_{\mathbb{A}_{f}}\ar[r]^{\eta\otimes \mathbb{A}_{f}}& T(X^{\tau})_{\mathbb{A}_{f}}\\ T(X)_{\mathbb{A}_{f}}\ar[u]^{r_{X}(s)}\ar[ur]_{\tau}.&}\]
\end{thm}
For a cubic fourfold $X$ over $\mathbb{C}$, let $\mathcal{A}_{X}\subset D^{b}(X)$ denote the Kuznetov component of the derived category \cite{ku}. Write $\widetilde{H}(\mathcal{A}_{X}, \mathbb{Z})$ for the Hodge structure introduced by Addington-Thomas \cite{add}. For a K3 surface $S$ and $\alpha\in {\rm Br}(S)$, the notation $\widetilde{H}(S, \alpha, \mathbb{Z})$ refers to the cohomology of the twisted K3 surface $(S, \alpha)$ defined by Huybrechts-Stellari \cite{Huy0}.
\begin{thm}(B\"ulles \cite[Theorem 0.4]{bu})
    Let $X\in \mathcal{C}_{d}$ be a special cubic fourfold. Assume  there exist a K3 surface $S,$ a Brauer class $\alpha\in {\rm Br}(S)$, and a Hodge isometry $\widetilde{H}(S, \alpha, \mathbb{Z})\cong \widetilde{H}(\mathcal{A}_{X}, \mathbb{Z})$. Then there is a cycle $Z\in {\rm CH}^{3}(S\times X)$ inducing an isomorphism of Chow groups ${\rm CH}_{0}(S)_{hom}\xrightarrow{\sim}{\rm CH}_{1}(X)_{hom}$ and hence an isomorphism of transcendental motives $t(S)(-1)\cong t(X)$. Furthermore, $h(X)\cong 1\oplus h(S)(-1)\oplus \mathbb{L}^{2}\oplus \mathbb{L}^{4},$ so $h(S)$ is finite-dimensional if and only if $h(X)$ is finite-dimensional.  
\end{thm}
Huybrechts \cite[Theorem 1.3]{Huy3} proved that for $X\in \mathcal{C}_{d}$, the existence of an isometry $\widetilde{H}(S, \alpha, \mathbb{Z})\cong \widetilde{H}(\mathcal{A}_{X}, \mathbb{Z})$ is equivalent to:
\\$\indent(\ast\ast^{'})$ 
For some $k, d_{0}, n\in \mathbb{Z}$,  $d_{0}|2n^{2}+2n+2$ and $d=k^{2}d_{0}$.
\\ According to Addington \cite[Theorem 1]{add0}, the existence of a K3 surface satisfying $H^{4}(X, \mathbb{Z}(2))\supset T(X)\cong T(S)\subset H^{2}(S, \mathbb{Z}(1))$ is equivalent to: 
\\$\indent (\ast\ast)$ 
For some $n\in\mathbb{Z}$, $d$ divides $2n^{2}+2n+2$.
\\Since $(\ast\ast)\Rightarrow(\ast\ast^{'})$, rank-21 cubic fourfolds satisfy   $(\ast\ast^{'})$ by Proposition \ref{21fun} (2). Hence:
\begin{cor}\label{singk3}
    For any rank-21 cubic fourfold $X$, there exist a singular K3 surface $S$ and a cycle $Z\in {\rm CH}^{3}(S\times X)$  suth that $Z$ induces an isomorphism of motives $t(S)(-1)\cong t(X)$. In particular, for every $\sigma\in {\rm Aut}(\mathbb{C})$, the cycle $Z^{\sigma}\in {\rm CH}^{3}(S^{\sigma}\times X^{\sigma})$ induces an isomorphism $t(S^{\sigma})(-1)\cong t(X^{\sigma})$.
   
\end{cor}
    For a prime $l$, let $H_{et, l}$ denote  $l$-adic \'etale cohomology. By (\ref{deck3}) and (\ref{decc4}),  for a K3 surface $S$  and a cubic fourfold $X$ over $\mathbb{C}$, we have the decompositions
    \begin{align}
       H^{2}(S, \mathbb{Q}_{l}(1))=&H^{2}_{et, l}(S)(1)={\rm NS}(S)_{\mathbb{Q}_{l}}\oplus T(S)_{\mathbb{Q}_{l}}
        \\ H^{4}(X, \mathbb{Q}_{l}(2))=&H^{4}_{et, l}(X)(2)=A(X)_{\mathbb{Q}_{l}}\oplus T(X)_{\mathbb{Q}_{l}}. 
    \end{align}
    \begin{cor}\label{action}Let $X$ be a rank-21 cubic fourfold,  and let $S$ be a singular K3 surface for which there exists an isomorphism of motives $t(S)(-1)\cong t(X)$ induced by some algebraic cycle $Z\in {\rm CH}^{3}(S\times X)_{\mathbb{Q}}$. Fix a prime $l$, and let $H_{et}$ denote  $l$-adic \'etale cohomology. For all $\sigma\in {\rm Aut}(\mathbb{C})$, the  diagram

        \[\xymatrix{T(S)_{\mathbb{Q}_{l}}\ar[r]^{\sim}\ar[d]_{\sigma}& T(X)_{\mathbb{Q}_{l}}\ar[d]^{\sigma}\\ T(S^{\sigma})_{\mathbb{Q}_{l}}\ar[r]^{\sim}& T(X^{\sigma})_{\mathbb{Q}_{l}}}\]
        is commutative.
    \end{cor}
    \begin{proof}The algebraic cycle $Z\in {\rm CH}^{3}(X\times S)$ determines a  morphism of motives $Z : h(X)\to h(S)(-1)$, which induces an isomorphism $t(X)\xrightarrow{\sim} t(S)(-1)$. 
        \\ \indent By \cite[Theorem 12. 7. 2]{ss}, we have a  commutative diagram
       
   \[\xymatrix{{\rm CH}^{3}(X\times S)\ar[rr]^{\gamma_{X\times S}(Z)}\ar[d]_{\sigma}&& H^{6}(X\times S)(3) \ar[d]^{\sigma}\\ {\rm CH}^{3}(X^{\sigma}\times S^{\sigma})\ar[rr]_{\gamma_{X^{\sigma}\times S^{\sigma}}(Z^{\sigma})} && H^{6}(X^{\sigma}\times S^{\sigma})(3).}\]
   Combining the K\"unneth formula, Poincare duality, and their functoriality yields a commutative diagram
   \[\xymatrix{H^{6}_{et, l}(X\times S)(3) \ar[d]_{\sigma}\ar[r]^{\sim}& \bigoplus_{i\in\mathbb{Z}} H^{8-i}_{et, l}(X)\otimes H^{i-2}_{et, l}(S)(-1)\ar[d]_{\sigma}\ar[r]^{\sim}& \bigoplus_{i\in \mathbb{Z}}{\rm Hom}(H^{i}_{et, l}(X), H^{i-2}_{et, l}(S)(-1))\ar[d]^{\sigma}\\ H^{6}_{et, l}(X^{\sigma}\times S^{\sigma})(3)\ar[r]^{\sim}& \bigoplus_{i\in \mathbb{Z}} H^{8-i}_{et, l}(X^{\sigma})\otimes H^{i-2}_{et, l}(S^{\sigma})(-1)\ar[r]^{\sim}& \bigoplus_{i\in \mathbb{Z}}{\rm Hom}(H^{i}_{et, l}(X^{\sigma}), H^{i-2}_{et, l}(S^{\sigma})(-1)).}\]
    Here ${\rm Hom}(-, -)$ denotes the homomorphisms of $\mathbb{Q}_{l}$-vector spaces, and the morphism ${\rm Hom}(H^{i}(X), H^{i-2}(S))\to {\rm Hom}(H^{i}(X), H^{i-2}(S)(-1))$  (sending $f\mapsto f^{'}$)  fits into the  commutative square
    \begin{align}\label{dia}\xymatrix{H^{4}_{et}(X, \mathbb{Q}_{l})\ar[r]^{f}\ar[d]_{\sigma}& H^{2}_{et}(S, \mathbb{Q}_{l}(-1))\ar[d]^{\sigma} \\H^{4}_{et}(X^{\sigma}, \mathbb{Q}_{l})\ar[r]_{f^{'}}& H^{2}_{et}(S^{\sigma}, \mathbb{Q}_{l}(-1)).}
    \end{align}
    Therefore $Z$ and $Z^{\sigma}$ satisfy  (\ref{dia}), yielding the claim.
    \end{proof}
    
    \begin{lemma}\label{r}
        Let $h_{i} : \mathbb{S}\to {\rm GL}(V_{i, \mathbb{R}})$ be $\mathbb{Q}$-HSs of K3 type of CM type.  Suppose there is  a Hodge isometry between $h_{1}$ and $h_{2}$ (hence $E\coloneq E(h_{1})=E(h_{2})\subset \mathbb{C}$). Then  for all $s\in \mathbb{A}^{\times}_{E}$  the  diagram 
        \[\xymatrix{V_{1, \mathbb{A}_{f}}\ar[r]^{\sim}&V_{2, \mathbb{A}_{f}}\\ V_{1, \mathbb{A}_{f}}\ar[u]^{r_{h_{1}}(s)}\ar[r]^{\sim}&V_{2, \mathbb{A}_{f}}\ar[u]_{r_{h_{2}}(s)}}\]
        is commutative.
    \end{lemma}
    \begin{proof}
        By assumption there is an isomorphism of $\mathbb{Q}$-tori
        ${\rm MT}(h_{1})\cong {\rm MT}(h_{2})$ making 
            \[\xymatrix{\mathbb{S}_{\mathbb{C}}\ar[r]^{h_{1}}\ar[dr]_{h_{2}}& {\rm MT}(h_{1})_{\mathbb{C}}\\  &{\rm MT}(h_{2})_{ \mathbb{C}}\ar[u] }\]
            commute. Passing to cocharacters gives
             \[\xymatrix{\mathbb{G}_{m, \mathbb{C}}\ar[r]^{\mu_{h_{1}}}\ar[dr]_{\mu_{h_{2}}}& {\rm MT}(h_{1})_{ \mathbb{C}}\\  &{\rm MT}(h_{2})_{ \mathbb{C}}.\ar[u] }\]
             Since the vertical morphism is defined over $\mathbb{Q}$, this  holds  over the reflex field $E$. Applying  ${\rm Res}_{E/\mathbb{Q}}$ and evaluating on ad\'eles yields the asserted commutativity.
    \end{proof}
   
    \begin{thm}(Rizov)\label{rizov}
Let $S/\mathbb{C}$ be a polarized K3 surface with complex multiplication, and let $E\subset \mathbb{C}$ be its reflex field. Let $\tau\in{\rm Aut}(\mathbb{C}/E)$ and $s\in \mathbb{A}^{\times}_{E}$ be an id\'ele with ${\rm art}_{E}(s)=\tau_{E^{ab}}$. Then there exists a unique Hodge isometry $\eta : T(S)_{\mathbb{Q}}\to T(S^{\tau})_{\mathbb{Q}}$ such that
\[\xymatrix{T(S)_{\mathbb{A}_{f}}\ar[r]^{\eta\otimes \mathbb{A}_{f}}& T(S^{\tau})_{\mathbb{A}_{f}}\\ T(S)_{\mathbb{A}_{f}}\ar[u]^{r_{S}(s)}\ar[ur]_{\tau}&}\]
commutes.
\end{thm}

\textbf{Proof of Theorem \ref{cmt}.}    Let $S/\mathbb{C}$ be the singular K3 surface from Corollary \ref{singk3}. Given $s\in\mathbb{A}^{\times}_{E}$, choose $\tau\in {\rm Aut}(\mathbb{C}/E)$ with  $\tau|_{E^{ab}}={\rm art}_{E}(s)$. By Theorem \ref{rizov}  there exists a unique Hodge isometry $\eta : T(S)_{\mathbb{Q}}\to T(S^{\tau})_{\mathbb{Q}}$ with 
\[\xymatrix{T(S)_{\mathbb{A}_{f}}\ar[r]^{\eta\otimes \mathbb{A}_{f}}& T(S^{\tau})_{\mathbb{A}_{f}}\\ T(S)_{\mathbb{A}_{f}}.\ar[u]^{r_{S}(s)}\ar[ur]_{\tau}&}\]
 Define the $\mathbb{Q}$-Hodge isometry $\eta_{1} : T(X)_{\mathbb{Q}}\to T(X^{\tau})_{\mathbb{Q}}$  by the commutative diagram 
\[\xymatrix{T(X)_{\mathbb{Q}}\ar[r]^{\sim}\ar[d]_{\eta_{1}}& T(S)_{\mathbb{Q}}\ar[d]^{\eta}\\ T(X^{\tau})_{\mathbb{Q}}\ar[r]^{\sim}& T(S^{\tau})_{\mathbb{Q}},}\]
where the horizontal isomorphisms come from $t(S)(-1)\cong t(X)$ and its $\tau$-conjugate (Corollary \ref{singk3}).
\\\indent Consider now 
  \[
\begin{tikzcd}[row sep=30pt,column sep=20pt]
T(X)_{\mathbb{A}_f}\arrow[rrr,"\sim"]\arrow[dd,"\eta_1\otimes\mathbb{A}_f"']&&&T(S)_{\mathbb{A}_f}\arrow[dd,"\eta\otimes\mathbb{A}_f"]\\
&T(X)_{\mathbb{A}_f}\arrow[r,"\sim"]\arrow[r,"\circlearrowright",yshift=10mm,phantom]\arrow[r,"\circlearrowright",yshift=-10mm,phantom]\arrow[ld,"\tau"]\arrow[lu,"r_X(s)"']&T(S)_{\mathbb{A}_f}\arrow[r,"\circlearrowright",xshift=1mm,phantom]\arrow[ru,"r_S(s)"]\arrow[rd,"\tau"']&\phantom{A}\\
T(X^{\tau})_{\mathbb{A}_f}\arrow[rrr,"\sim"']&&&T(S^{\tau})_{\mathbb{A}_f}.
\end{tikzcd}   
   \]
   By construction,  the right triangle commutes; by Lemma \ref{r} and Corollary \ref{action}, the upper and lower trapezoids commute; hence the left triangle commutes as well. This proves Theorem \ref{cmt}. 
    \begin{cor}\label{pcmt}
Let $X/\mathbb{C}$ be a rank-21 cubic fourfold with  reflex field $E\subset \mathbb{C}$. Let $\tau\in{\rm Aut}(\mathbb{C}/E)$ and $s\in \mathbb{A}^{\times}_{E}$ be an id\'ele with ${\rm art}_{E}(s)=\tau_{E^{ab}}$. Then there exists a unique Hodge isometry $\eta : PH^{4}(X, \mathbb{Q})(2)\to PH^{4}(X^{\tau}, \mathbb{Q})(2)$ such that
\[\xymatrix{PH^{4}(X)(2)_{\mathbb{A}_{f}}\ar[r]^{\eta\otimes \mathbb{A}_{f}}& PH^{4}(X^{\tau})(2)_{\mathbb{A}_{f}}\\ PH^{4}(X)(2)_{\mathbb{A}_{f}}\ar[u]^{r_{X}(s)}\ar[ur]_{\tau}&}\]
     commutes.
    \end{cor}
    \begin{proof}
        For $s\in \mathbb{A}^{\times}_{E}$ and $\tau$ as above, take the $\mathbb{Q}$-Hodge isometry $\eta : T(X)_{\mathbb{Q}}\to T(X^{\tau})_{\mathbb{Q}}$ from Theorem \ref{cmt}. Moreover,  the pullback $\tau^{\ast} : {\rm CH}^{2}(X)_{\mathbb{Q}}\to {\rm CH}^{2}(X^{\tau})_{\mathbb{Q}}$ sends $c_{1}(\mathcal{O}_{X}(1))^{2}$ to $c_{1}(\mathcal{O}_{X^{\tau}}(1))^{2}$. Hence we obtain a $\mathbb{Q}$-Hodge isometry $\eta_{1}:PH^{4}(X, \mathbb{Q})(2)\to PH^{4}(X^{\tau}, \mathbb{Q})(2)$. Since $r_{h}(s)$  acts trivially on the algebraic part of $PH^{4}(X, \mathbb{Q})(2)$  (Example \ref{tri}), the diagram commutes. 
    \end{proof}
     \subsection{The proof of Theorem 1.4}
     In this section we complete  the proof of the main theorem (Theorem \ref{main3}).
     \\ $\textbf{Step 1}.$ For any imaginary quadratic field $E\subset \mathbb{C}$,  the period map $j_{N, \mathbb{C}} : \widetilde{\mathcal{C}^{[N]}}_{\mathbb{C}}\to {\rm Sh}_{K_N}(L)_{\mathbb{C}}$ is defined over $E$.
     \\\indent Take a point $(X, h, \alpha)\in \widetilde{\mathcal{C}^{[N]}}({\mathbb{C}})$, where $X$ is a rank-21 cubic fourfold $X$ with  reflex field $E\subset\mathbb{C}$. Fix  $\sigma\in {\rm Aut}(\mathbb{C}/E)$. We claim that  
     \begin{align}\label{com j}
         j_{N, \mathbb{C}}(\sigma(X, h, \alpha))=\sigma(j_{N, \mathbb{C}}(X, h, \alpha))
     \end{align}
     in ${\rm Sh}_{K_{N}}(L)(\mathbb{C})$.
     Write  $\sigma(X, h, \alpha)=(X^{\sigma}, h^{ \sigma}, \alpha^{\sigma})$. 
     \\\indent Let $\widetilde{\alpha} : L_{0, \mathbb{Z}_{\mathscr{B}}}\to H^{4}_{et}(X, \mathbb{Z}_{\mathscr{B}}(2))$ represent $\alpha$, and choose an isometry $a : H^{4}(X, \mathbb{Z}(2))\to L_{0}$ with $a(h^{2})=v$ and $a\circ \widetilde{\alpha}\in {\rm SO}(V)(\hat{\mathbb{Z}})$. By Corollary \ref{pcmt}, there exists a $\mathbb{Q}$-Hodge isometry $\eta : PH^{4}(X, \mathbb{Q}(2))\to PH^{4}(X^{\sigma}, \mathbb{Q}(2))$ such that $(\eta\otimes\mathbb{A}_{f})\circ r_{X}(s)=\sigma^{\ast}$ for some id\'ele $s\in \mathbb{A}^{\times}_{E}$. Consider 
     \begin{align}
         a_{X^{\sigma}}:=a_{X}\circ \eta^{-1}: PH^{4}(X^{\sigma}, \mathbb{Q}(2))\to V.
     \end{align}
     Then, by Lemma \ref{conja},
     \begin{align}
         a_{X^{\sigma}}\circ \widetilde{\alpha}^{\sigma}&=a_{X}\circ \eta^{-1}\circ \sigma^{\ast}\circ \widetilde{\alpha} \notag \\
         &=a_{X}\circ r_{X}(s)\circ a_{X}^{-1}\circ a_{X}\circ \widetilde{\alpha} \notag \\
         &=r_{(a_{X}\circ h_{X}\circ a_{X}^{-1})}(s)\circ a_{X}\circ \widetilde{\alpha} \in{\rm SO}(V)(\mathbb{A
         }_{f}).\notag 
     \end{align}
     By Remark \ref{mod1},  
     \begin{align}
         j_{N, \mathbb{C}}([X^{\sigma}, h^{\sigma},\alpha^{\sigma}])&=[a_{X^{\sigma}}\circ h_{X^{\sigma}}\circ a^{-1}_{X^{\sigma}}, a_{X^{\sigma}}\circ \widetilde{\alpha}^{\sigma}] \notag \\
         &=[a_{X}\circ (\eta^{-1}\circ h_{X^{\sigma}}\circ\eta)\circ a_{X}^{-1}, a_{X}\circ \alpha^{-1}\circ \sigma^{\ast}\circ \widetilde{\alpha}] \notag \\
         &=[a_{X}\circ h_{X}\circ a_{X}^{-1}, r_{(a_{X}\circ h_{X}\circ a_{X}^{-1})}(s)\circ a_{X}\circ \widetilde{\alpha}] \notag \\
         &=\sigma(j_{N, \mathbb{C}}(X, h, \alpha)). \notag 
     \end{align}
     This proves (\ref{com j}).
    \\ $\textbf{Step 2}.$ By Theorem \ref{apl1}, the points $(X, h, \alpha)$ with  $X$ rank-21 and reflex field $E\subset \mathbb{C}$ form a Zariski-dense set of $\widetilde{\mathcal{C}^{[N]}}$. Hence the morphism $j_{N, \mathbb{C}}$ is defined over $E\subset \mathbb{C}$.
    \\$\textbf{Step 3. }$ If $E_{1}$ and $E_{2}$ are distinct imaginary quadratic field, Step 2 shows the period map is defined over both $E_{1}$ and $E_{2}$, hence over their intersection $\mathbb{Q}=E_{1}\cap E_{2}$. This completes the proof of Theorem \ref{main3}.
    \subsection{Application of Theorem \ref{main3}: arithmetic of cubic fourfolds of CM type}
   
    \begin{cor}(Complex multiplication for cubic fourfolds)\label{cmtgen}
        Let $X/\mathbb{C}$ be a cubic fourfold with reflex field $E\subset \mathbb{C}$. Let $\sigma\in{\rm Aut}(\mathbb{C}/E)$ and $s\in \mathbb{A}^{\times}_{E}$ be an id\'ele with ${\rm art}_{E}(s)=\sigma |_{E^{ab}}$. Then there exists a unique Hodge isometry $\eta : PH^{4}(X, \mathbb{Q})(2)\to PH^{4}(X^{\sigma}, \mathbb{Q})(2)$ such that the diagram
\[\xymatrix{PH^{4}(X)(2)_{\mathbb{A}_{f}}\ar[r]^{\eta\otimes \mathbb{A}_{f}}& PH^{4}(X^{\sigma})(2)_{\mathbb{A}_{f}}\\ PH^{4}(X)(2)_{\mathbb{A}_{f}}\ar[u]^{r_{h}(s)}\ar[ur]_{\sigma}&}\]
commutes.
      
    \end{cor}
    \begin{proof}
        Choose a level-13 structure $\alpha$ on $(X, h)$; then $(X, h, \alpha)\in\widetilde{\mathcal{C}^{[13]}}(\mathbb{C})$. Let $\widetilde{\alpha}$ represent $\alpha$. Choose markings $a_{X} : P^{4}(X, \mathbb{Z}(2))\to L$ and $a_{X^{\sigma}} : PH^{4}(X, \mathbb{Z})(2)\to L$ with $a_{X}\circ \widetilde{\alpha}\in {\rm SO}(V)(\hat{\mathbb{Z}})$ and $a_{X^{\sigma}}\circ \widetilde{\alpha}^{\sigma}\in {\rm SO}(V)(\hat{\mathbb{Z}})$ (existence  follows from the proof of Theorem \ref{main3} and Remark \ref{mod1}). 
        \\\indent By Theorem \ref{main3} and
        the  definition of canonical models,
        \begin{align}
             [(a_{X}\circ h_{X}\circ a_{X}^{-1}, (a_{X}\circ r_{X}(s)\circ a_{X}^{-1})\circ (a_{X}\circ \widetilde{\alpha}))]=[(a_{X^{\sigma}}\circ h_{X^{\sigma}}\circ a^{-1}_{X^{\sigma}}, a_{X^{\sigma}}\circ \widetilde{\alpha}^{\sigma})] \notag 
        \end{align}
        in ${\rm Sh}_{K_{13}}(L)(\mathbb{C})$.
        Hence there exists $q\in G(\mathbb{Q})$ with 
        \begin{align}
            (q\cdot (a_{X}\circ h_{X}\circ a_{X}^{-1}), q\cdot (a_{X}\circ r_{X}(s)\circ \widetilde{\alpha}))=(a_{X^{\sigma}}\circ h_{X^{\sigma}}\circ a^{-1}_{X^{\sigma}}, a_{X^{\sigma}}\circ \sigma\circ \widetilde{\alpha}). \notag 
        \end{align}
       Define the Hodge isometry $\eta : PH^{4}(X, \mathbb{Q}(2))\to PH^{4}(X^{\sigma}, \mathbb{Q}(2))$ by $\eta :=a^{-1}_{X^{\sigma}}\circ q\circ a_{X}$. Then 
        \begin{align}
            a_{X^{\sigma}}\circ (\eta\otimes \mathbb{A}_{f})\circ r_{X}(s)= q\circ a_{X}\circ r_{X}(s)=a_{X^{\sigma}}\circ \sigma,\notag
        \end{align}
       so $(\eta\otimes \mathbb{A}_{f})\circ r_{X}(s)=\sigma$.
    \end{proof}
    Let   $\overline{\mathbb{Q}}$ denote the algebraic closure of $\mathbb{Q}$ in $\mathbb{C}$.
    \begin{cor}\label{abelian ext}
        Let $X$ be a cubic fourfold of CM type with reflex field $E\subset \overline{\mathbb{Q}}$. Then there exists an abelian extension $k/E$ such that $X$ is defined over $k$.
    \end{cor}
    \begin{proof}
        Choose a level-13 structure. We have an open immersion over $\mathbb{Q}$:
        \begin{align}j_{13, \mathbb{Q}} : \widetilde{\mathcal{C}^{[13]}}_{\mathbb{Q}}\to {\rm Sh}_{K_{13}}(L)_{\mathbb{Q}}. \notag 
        \end{align}
        For a level-13 structure $\alpha$ on $(X, h)$, the point $j_{13, \mathbb{C}}(X, h, \alpha)$ is  special in ${\rm Sh}_{K_{13}}(L)(\mathbb{C})$. 
        \\\indent By the definition of the canonical models, this point $j_{13, \mathbb{C}}(X, h, \alpha)$ is $E^{ab}$-rational in ${\rm Sh}_{K_{13}}(L)_{\mathbb{C}}$. Since $j_{13, \mathbb{C}}$ is defined over $\mathbb{Q}$, the point lies in $\widetilde{\mathcal{C}^{[13]}}(k)$ for some  abelian extension $k/E$ contained in $\overline{\mathbb{Q}}$. Thus there exists a cubic fourfold $X^{'}$ defined over $k$ with $X=X^{'}\times_{k}\mathbb{C}$.\end{proof}
        \section{Modularity of rank-21 cubic fourfolds over $\mathbb{Q}$} 
        \subsection{Reformulation of the complex multiplication theory for rank-21 cubic fourfolds}\label{galoisrep}
Let $X/\mathbb{C}$ be a rank-21 cubic fourfold with the reflex field $K\subset\mathbb{C}$. Suppose that $X$ is defined over $\mathbb{Q}$, and  fix a model $X_{0}$  over $\mathbb{Q}$. Let $\overline{\mathbb{Q}}$ denote the algebraic closure of $\mathbb{Q}$ in $\mathbb{C}$. For every prime $l$, there is a  canonical isomorphism
\begin{align}\label{basechange}
   H^{4}_{et}(X, \mathbb{Z}_{l}(2))\cong H^{4}_{et}(X_{0, \overline{\mathbb{Q}}}, \mathbb{Z}_{l}(2)).
\end{align} 
Via (\ref{basechange}), the Galois group ${\rm Gal}(\overline{\mathbb{Q}}/\mathbb{Q})$ acts on $H^{4}_{et}(X, \mathbb{Z}_{l}(2))$. Since (\ref{basechange}) induces the canonical ${\rm Gal}(\overline{\mathbb{Q}}/\mathbb{Q})$-equivariant isomorphism 
\begin{align}
    A(X)_{\mathbb{Z}_{l}}\cong A(X_{0, \overline{\mathbb{Q}}})_{\mathbb{Z}_{l}},\notag 
\end{align} we obtain a  2-dimensional $l$-adic representation 
\begin{align}
    \rho_{l}: {\rm Gal(\overline{\mathbb{Q}}/\mathbb{Q})}\to {\rm GL}(T(X)_{\mathbb{Z}_{l}}).
\end{align}Taking the product over all prime $l$ (and extending  scalars to $\mathbb{Q}$), we obtain an $\mathbb{A}_{f}$-representation  \begin{align}\label{galdef}
    \rho:{\rm Gal}(\overline{\mathbb{Q}}/\mathbb{Q})\to {\rm GL}(T(X)_{\mathbb{A}_{f}}).
\end{align}
\begin{thm}\label{refcmt}Let $X/\mathbb{C}$ be a rank-21 cubic fourfold with the reflex field $K$. Suppose that $X$ has a model $X_{0}$ over $\mathbb{Q}$. Let $G$ be the Mumford-Tate group of $T(X)$. Then: 
\\(1) The restriction $\rho|_{{\rm Gal}(\overline{\mathbb{Q}}/K)}$ factors as 
\begin{align}
    {\rm Gal}(\overline{\mathbb{Q}}/K)\to {\rm Gal}(K^{ab}/K)\xrightarrow{\rho_{0}}G(\mathbb{A}_{f})\hookrightarrow {\rm GL}(T(X)_{\mathbb{A}_{f}}),\notag 
\end{align}
for  a continuous homomorphism $\rho_{0}$.
\\ (2)   The diagram 
    \begin{align}\label{galoiscmt}\xymatrix{\mathbb{A}^{\times}_{K}\ar[rr]^{{\rm art}_{K}}\ar[d]_{r_{X}}&& {\rm Gal}(K^{ab}/K)\ar[d]^{\rho_0} 
    \\ G(\mathbb{A}_{f})\ar[rr]&&G(\mathbb{A}_{f})/G(\mathbb{Q})}\end{align}
    is commutative. 
\end{thm}
\begin{proof}
    (1) We first  show that the image $\rho({\rm Gal}(\overline{\mathbb{Q}}/K))$ is contained in $G(\mathbb{A}_{f})$. Fix $\sigma\in {\rm Gal}(\overline{\mathbb{Q}}/K)$. Choose an automorphism $\overline{\sigma}\in {\rm Aut}(\mathbb{C}/K)$ such that  $\overline{\sigma}|_{\overline{\mathbb{Q}}}=\sigma$. Let $s\in \mathbb{A}^{\times}_{K}$ be an id\'ele with ${\rm art}_{K}(s)=\sigma|_{K^{ab}}$. By Theorem \ref{cmt}, there exists a $\mathbb{Q}$-Hodge isometry $\eta\in G(\mathbb{Q})$ such that $\rho(\sigma)=\eta\circ (r_{X}(s))$. Since $r_{X}(s)\in G(\mathbb{A}_{f})$, it follows that $\rho(\sigma)\in G(\mathbb{A}_{f})$.
    \\\indent By Theorem \ref{zrfun}, $G(\mathbb{A}_{f})$ is commutative, and hence $\rho$ factors through ${\rm Gal}(K^{ab}/K)$. This yields a continuous homomorphism $\rho_{0}$.
    \\(2) By the argument above, the element    $\rho(\sigma)$ and $r_{X}(s)$ differ by  an element $\eta\in G(\mathbb{Q})$. This is precisely the statement that the diagram (\ref{galoiscmt}) commutes. 
\end{proof}
 
\begin{cor}\label{u}
    Under the assumptions of Theorem \ref{refcmt}, there exists a unique locally constant homomorphism 
    \begin{align}
        u:\mathbb{A}^{\times}_{K}\to G(\mathbb{Q})=\{x\in K^{\times}| ~x\overline{x}=1\}\subset K^{\times} \notag 
    \end{align}
   such that  $\rho_{0}\circ {\rm art_{K}}=u\cdot  r_{X}$,
   where $\cdot$ dentes the product in $G(\mathbb{A}_{f})$.
\end{cor}
\begin{proof}
    This is just a reformulation of the commutativity of the diagram (\ref{galoiscmt}).
\end{proof}

        \subsection{Construction of the algebraic Hecke character $\psi$}
        Let $X/\mathbb{C}$ be a rank-21 cubic fourfold  with reflex field $K\subset\mathbb{C}$. Suppose that $X$ is defined over $\mathbb{Q}$, and  let $\overline{\mathbb{Q}}$ be the algebraic closure of $\mathbb{Q}$ in $\mathbb{C}$. In particular, $K\subset \overline{\mathbb{Q}}$. 
        \\\indent Let  $D$ be the positive integer such that $-D$ is the discriminat of $K$. Define
        \begin{align}\varphi:=\Big(\frac{-D}{-}\Big) \notag 
        \end{align}
        to be the quadratic Dirichlet character associated with $K$, given by Kronecker symbol.
        \begin{prop}\label{frob}For every  prime number $p\notmid D$, the following conditions are equivalent:
        \\(1) $\varphi(p)=1$.
        \\(2) $p$ splits in $K$. 
        \\(3) The geometric Fronenius ${\rm Frob}_{p}\in {\rm Gal}(\overline{\mathbb{Q}}/\mathbb{Q})$ lies in ${\rm Gal}(\overline{\mathbb{Q}}/K)$.
             \end{prop}
             \begin{proof}
                 The equivalence $(1)\Leftrightarrow (2)$ is well known (for example, see \cite[Chapter 3, Section 8, Theorem 1]{bosh}). Since ${\rm Gal}(K/\mathbb{Q})=\{\pm 1\}$ and ${\rm Frob}_{p}^{-1}|_{K}=\Big(\frac{K/\mathbb{Q}}{p}\Big)$, it suffices to show that 
                 \begin{align}
                     \Big(\frac{K/\mathbb{Q}}{p}\Big)=1\Leftrightarrow (1), \notag 
                 \end{align}
                 where $\Big(\frac{K/\mathbb{Q}}{p}\Big)\in {\rm Gal}(K/\mathbb{Q})$ denotes the Artin symbol. 
                 \\\indent Indeed, let $\mathfrak{p}$ be a prime of  $K$ lying over $p$, and let $D(\mathfrak{p}/p)$ be the decomposition group. Then $\Big(\frac{K/\mathbb{Q}}{p}\Big)\in {\rm Gal}(K/\mathbb{Q})$ is an inverse image of the arthmetic Frobenius$x\mapsto x^{p}$ under the canonical surjection $D(\mathfrak{p}/p)\to {\rm Gal}(k(\mathfrak{p})/\mathbb{F}_{p})$, where $k(\mathfrak{p})$ is the residue field. Since $p\notmid D$, the prime $p$ is unramified in the extension $K/\mathbb{Q}$, so  this surjection  is in fact an isomorphism. If  $f$ denotes the residue degree of  $\mathfrak{p}|p$, then  $\# D(\mathfrak{p}/p)=f$. Hence 
                 \begin{align}\Big(\frac{K/\mathbb{Q}}{p}\Big)=1 \Leftrightarrow f=1, \notag 
                 \end{align}which is exactly condition (1). 
             \end{proof}
      \begin{prop}\label{valloni}Let $X/\mathbb{C}$ be a rank-21 cubic fourfold with  its reflex field $K\subset\mathbb{C}$. Let  $r_{X}:\mathbb{A}^{\times}_{K}\to {\rm MT}(T_{X})(\mathbb{A}_{f})\subset {\rm SO}(T_{X})(\mathbb{A}_{f})$ be the homomorphism defined in  \S \ref{rx}. For every $s\in \mathbb{A}^{\times}_{E}$, the action $r_{X}(s): T_{X, \mathbb{A}_{f}}\to T_{X, \mathbb{A}_{f}}$ is given by the multiplication by $\overline{s_{f}}/s_{f}$, where $s_{f}$ is the finite part of $s$ and $\overline{s}$ denotes the complex conjugation.
      \end{prop}
      \begin{proof}
      Let $\mu_{h}:\mathbb{G}_{m, K}\to MT(T(X))_{K}$ be the cocharacter associated with the $\mathbb{Q}$-HS $T(X)$. By the definition of $\mu_{h}$ and the inclusion ${\rm MT}(T(X))\subset {\rm Res}_{K/\mathbb{Q}}(\mathbb{G}_{m, K})$, the map $\mu_{h}$ is described by 
      \begin{align}
   \mathbb{G}_{m}(K)=  K^{\times}\ni     z\mapsto (z^{-1}, z)\in K^{\times}\times K^{\times}={\rm Res}_{K/\mathbb{Q}}(\mathbb{G}_{m, K})(K).     \notag  \end{align}
  Hence   $r({T(X), \mu_{h}})$ is given by:
    \begin{align}
        {\rm Res}_{K/\mathbb{Q}}(\mathbb{G}_{m, K})(\mathbb{Q})=K^{\times}\ni z\mapsto \overline{z}/z=N_{K/\mathbb{Q}}(\mu_{h}(z))\in \{x\in K^{\times}|~x\overline{x}=1\}={\rm MT}(T(X))(\mathbb{Q}), \notag 
    \end{align}
    as required. 
      \end{proof}
      \begin{prop}\label{tatetwist}Let $l$ be a prime, and let 
      \begin{align}\chi_{l}:{\rm Gal}(\overline{\mathbb{Q}}/\mathbb{Q})\to\mathbb{Q}_{l}^{\times} \notag 
      \end{align} be the $l$-adic cyclotomic character. If $p\neq l$ is a prime, then $\chi_{l}$ is unramifeid at $p$, and  
\begin{align}\chi_{l}({\rm Frob}_{\mathfrak{p}})=1/p. \notag 
\end{align}
In particular, after fixing an isomorphism $\overline{\mathbb{Q}_{l}}\cong \mathbb{C}$, $\chi_{l}$ corresponds to the algebraic Hecke character of infinite type $(-1, -1)$:
      \begin{align}
          |N_{K/\mathbb{Q}}(-)|: \mathbb{A}^{\times}_{K}\xrightarrow{N_{K/\mathbb{Q}}}\mathbb{A}^{\times}_{\mathbb{Q}}\xrightarrow{|-|_{\mathbb{A}_{\mathbb{Q}}}}\mathbb{R}^{\times}_{>0}. \notag
      \end{align} 
      \end{prop}
      \begin{proof}We may assume $l\neq 2$. 
          Let $\zeta_{l}\in\overline{\mathbb{Q}}$ be a primitive $l$-th root of unity.  Then the discriminant of $\mathbb{Q}(\zeta_{l})$ is $(-1)^{(l-1)/2}l^{l-2}$. If $p\neq l$, then $p$ is unramified in $\mathbb{Q}(\zeta_{l})$. Since the restriction of $\chi_{l}$ to ${\rm Gal}(\overline{\mathbb{Q}}/\mathbb{Q}(\zeta_{l}))$ is trivial, the charcter $\chi_{l}$ is unramified at $p$. 
          \\\indent Let $\mu_{l^{n}}(\mathbb{\overline{Q}})$ be the group of $l^{n}$-th roots of unity in $\overline{\mathbb{Q}}$.  Galois theory gives
          \begin{align}
              {\rm Gal}(\mathbb{Q}(\zeta_{l})/\mathbb{Q})\cong (\mathbb{Z}/l\mathbb{Z})^{\times}, \notag 
          \end{align} and the Artin symbol $\big(\frac{\mathbb{Q}(\zeta_{l})/\mathbb{Q}}{p}\big)$ corresponds to the arithmetic Frobenius. Thus 
          \begin{align}\chi_{l}({\rm Frob}_{p}^{-1})=p. \notag 
          \end{align}
      \end{proof}
       \begin{thm}\label{hecke}
          Fix an isomorphism $\iota: \overline{\mathbb{Q}_{l}}\cong \mathbb{C}$ of fields.  Then there exists a unique algebraic Hecke character $\psi:\mathbb{A}^{\times}_{K}\to \mathbb{C}$ such that 
          \\(1) the infnite type of $\psi$ is $(2, 0)$; 
\\(2)  For every prime $p\notmid lDM$, the representation $\rho_{l}$ is unramified at $p$, where $M$ is the norm of the conductor of $\psi$. Moreover, if $\mathfrak{p}| p$, then \begin{align}\label{trace}
            {\rm Tr}(  \rho_{l}(1)({\rm Frob}_{p}))&=\begin{cases}
\chi(u_\mathfrak{p})+\chi(\overline{u_{\mathfrak{p}}}), &p\textnormal{ : split},
                \\ 0, &p:\textnormal{inert},
            \end{cases}
             \\  \label{det}{\rm det}(\rho_{l}(1)({\rm Frob}_{p}))&=\begin{cases}
\chi(u_\mathfrak{p})\chi(\overline{u_{\mathfrak{p}}}),&p\textnormal{ : split},
                     \\ -\chi(u_\mathfrak{p}), &p:\textnormal{inert},
                 \end{cases}
             \end{align} 
             where $\rho_{l}(1)=\rho_{l}\otimes\chi_{l}$ is the Tate twist  and  $u_{\mathfrak{p}}\in \mathcal{O}_{K, \mathfrak{p}}$ is a uniformizer.
       \end{thm}
       \begin{proof}
            We first construct an algebraic Hecke character $\psi_{0}$  over $K$ of infinite type $( 1, -1)$ satisfying (\ref{trace}), (\ref{det}) for ${\rm Tr}(\rho_{l}({\rm Frob}_{p}))$ and ${\rm det}(\rho_{l}({\rm Frob}_{p}))$, respectively.
           \\\indent Define a continuous character $r : \mathbb{A}^{\times}_{K}\to \mathbb{C}^{\times}$ by 
           \begin{align}  \mathbb{A}^{\times}_{K}\xrightarrow{\overline{s}/s}\mathbb{A}^{\times}_{K}\xrightarrow{\textnormal{projection}}\mathbb{C}^{\times}. \notag 
           \end{align}
          Let $u:\mathbb{A}^{\times}_{K}\to K^{\times}$ be the locally constant homomorphism constructed in Corollary \ref{u}.  The product 
          \begin{align}\psi_{0}:=ur\notag
          \end{align}is an algebraic Hecke character of infinite type $(1, -1)$. For every $p\notmid lDM$ and $\mathfrak{p}|p$, we will show that $\psi_{0} (u_{\mathfrak{p}})$ satisfies (\ref{trace}) and (\ref{det}) for $\rho_{l}$.
          \\\indent Let $p\notmid lDM$  and  $\mathfrak{p}\subset K$ be a prime  above $p$. By Corollary \ref{u}, we have 
          \begin{align}\label{rhoar}
              \rho_{l}\circ {\rm art}_{K}=u
          \end{align}
          on $K_{\mathfrak{p}}$. Since $p\notmid M$, the character $\psi_{0}$ is unramified at $\mathfrak{p}$. By (\ref{rhoar}), it follows that $\rho_{l}$ is unramified at $\mathfrak{p}$. Because  $p\notmid D$, the prime $p$ is unramified in $K/\mathbb{Q}$, hence $\rho_{l}$ is unramified at $p$. 
            \\\indent  If $p$ splits, i.e. $(p)=\mathfrak{p}\overline{\mathfrak{p}}$, then  \begin{align}\label{norm}p^{-1}=\pi_{\mathfrak{p}}\overline{\pi_{\mathfrak{p}}}=N_{K_{\mathfrak{p}}/\mathbb{Q}_{l}}(\pi_{\mathfrak{p}}). 
           \end{align}
         By Proposition \ref{frob}, ${\rm Frob}_{p}\in{\rm Gal}(\overline{\mathbb{Q}}/K)$. Thus the relation (\ref{norm}) implies ${\rm art}_{K}(u_{\mathfrak{p}})={\rm Frob}_{p}|_{{\rm Gal}(K^{ab}/K)}$. Since $\mathfrak{p}\notmid l$, the $l$-component of $(\pi_{\mathfrak{p}}/\overline{\pi_{\mathfrak{p}}})_{f}$ is just 1. By (\ref{rhoar}), 
         \begin{align}
             \rho_{l}({\rm Frob}_{p})=\psi_{0}(u_{\mathfrak{p}}). \notag 
         \end{align}
          On the other hand, we have  $1=\psi_{0, \infty}^{-1}(p)=\psi_{0, fin}(p)=\psi_{0}(u_{\mathfrak{p}})\psi_{0}(\overline{u_\mathfrak{p}})$, where $\psi_{0, fin}$ and $\psi_{0, \infty}$ denote the finite and  infinite part of $\psi_{0}$, respectively. Thus we obtain $\overline{\psi_{0}(u_{\mathfrak{p}})}=\psi_{0}({\overline{u_{\mathfrak{p}}}})$.   Hence $\psi_{0}(u_{\mathfrak{p}})$ and $\psi_{0}(\overline{u_{\mathfrak{p}}})$ are the eigenvalues of $\rho_{l}({\rm Frob}_{p})$, giving (\ref{trace}) and (\ref{det}).
          \\\indent If $p$ is inert, then,  by Proposition \ref{frob}, ${\rm Frob}_{p}^{2}\in{\rm Gal}(\overline{\mathbb{Q}}/K)$. Since \begin{align}\notag
p^{-2}=\pi_{\mathfrak{p}}^{2}=N_{K_{\mathfrak{p}}/\mathbb{Q}_{l}}(\pi_{\mathfrak{p}}), 
          \end{align}
          we obtain ${\rm art}_{K}(u_{\mathfrak{p}})={\rm Frob}_{p}^{2}|_{{\rm Gal}(K^{ab}/K)}$. Using (\ref{rhoar}), we find
          \begin{align}
              \rho_{l}({\rm Frob}_{p})^{2}=\rho_{l}({\rm Frob}_{p}^{2})=\psi_{0}(u_{\mathfrak{p}}). \notag
          \end{align}
         Thus the characteristic polynomial of $\rho_{l}({\rm Frob}_{p})$ is  
         \begin{align}
             T^{2}-\psi_{0}(u_{\mathfrak{p}}),\notag
             \end{align}and the coefficients give  (\ref{trace}) and (\ref{det}). 
         \\\indent Finally, define
         \begin{align}\notag \psi:=|N_{K/\mathbb{Q}}(-)|\cdot \psi_{0}:\mathbb{A}^{\times}_{K}\to \mathbb{C}^{\times}.
         \end{align} Since $|N_{K/\mathbb{Q}}(-)|$ has infinite type $(-1, -1)$, the infinite type of $\psi$ is $(2, 0)$. By Proposition \ref{tatetwist}, $|N_{K/\mathbb{Q}}(-)|$ corresponds to the $l$-adic cyclotomic character $\chi_{l}$. Hence $\psi$ satisfies (\ref{trace}) and (\ref{det}) for $\rho_{l}(1)$.
       \end{proof}
       
       \subsection{The proof of Theorem \ref{modular}} In this section  we give the proof of Theorem \ref{modular}.
        \\\indent Let $\mathfrak{m}\subset \mathcal{O}_{K}$ be the conductor of $\psi$. Then $\psi$ corresponds to the classical Hecke character
        \begin{align}
         \widetilde{\psi}:   \{\textnormal{fractional ideals of}~K~\textnormal{prime to}~\mathfrak{m}\}\to \mathbb{C}^{\times},\notag 
         ~\mathfrak{p}\mapsto\psi(u_{\mathfrak{p}}),
        \end{align}
        where $u_\mathfrak{p}\subset\mathcal{O}_{K, \mathfrak{p}}$ is a uniformizer. Since $\psi$ has infinite type $(2, 0)$, we have 
        \begin{align}\label{clasic}
            \widetilde{\psi}((\alpha))=\psi_{fin}(\alpha)=\psi^{-1}_{\infty}(\alpha)=\alpha^{2}, ~ (\alpha\in K^{\times}, \alpha\equiv1~\textnormal{mod}~ \mathfrak{m}).
        \end{align}
        Let $M$ be the norm of $\mathfrak{m}$, and define a Dirichlet character    
        \begin{align}
        \eta:(\mathbb{Z}/M\mathbb{Z})^{\times}\ni    a\mapsto \widetilde{\psi}((a))/a^{2}\in\mathbb{C}.
        \end{align} By the relation (\ref{clasic}), the character $\eta$ is well-defined. Put $\varepsilon=\eta\varphi$, where $\varphi$ is the quadratic Dirichlet character associated with $K/\mathbb{Q}$. 
        \\\indent Define the $q$-series $g=\sum^{\infty}_{n=1}c_{n}q^{n}$ by
        \begin{align}
            g=\sum_{\mathfrak{a}\subset\mathcal{O}_{K}:~\textnormal{ideal prime to }~\mathfrak{m}}\widetilde{\psi}(\mathfrak{a})q^{N(\mathfrak{a})},~(q=e^{2\pi i\tau})
        \end{align}
        where $N(\mathfrak{a})$ is the norm of $\mathfrak{a}$.
        \begin{thm}(Shimura \cite[Lemma 3]{sh})
            The series $g$ is a cusp form of weight $3$ and character $\varepsilon$ on $\Gamma_{1}(DM)$.
        \end{thm}
        
   \begin{thm}\label{coef}(Ribet \cite[Corollary 3.5]{rib})
       There exists a unique newform $f=\sum^{\infty}_{n=1}a_{n}q^{n}$ of weight $3$, character $\varepsilon$, and level dividing DM such that 
       \begin{align}a_{p}=c_{p}\notag 
       \end{align}for all $p\notmid DM$. Moreover, $a_{p}=0$ when $\varphi(p)=-1$; in particular, $f$ has complex multiplication by $\varphi$. 
   \end{thm}
   Via a fixed isomorphism $\iota:\overline{\mathbb{Q}_{l}}\cong \mathbb{C}$, the $\mathbb{Q}_{l}$-algebra  $K\otimes \mathbb{Q}_{l}$ embeds in $\overline{\mathbb{Q}_{l}}$. 
   \begin{thm}\label{ribet}(Ribet \cite[Theorem 2.1 and Corollary 2.4]{rib}) Let $l$ be a prime. Then there exists a continuous representation 
       $\rho_{f, l} : {\rm Gal}(\overline{\mathbb{Q}}/\mathbb{Q})\to {\rm GL}_{2}(K\otimes \mathbb{Q}_{l})$ with the following property:
       \\(1) If $p\notmid lN$, then $\rho_{l}$ is unramified at $p$, where $N$ is the level of $f$.
       \\(2) If $p\notmid lN$, then 
       \begin{align}
           {\rm Tr}(\rho_{f, l}({\rm Frob}_{p}))=c_{p}, ~~~
           {\rm det}(\rho_{f, l}({\rm Frob}_{p}))=\varepsilon(p)p^{2}. 
       \end{align}
       Such a representation $\rho_{f, l}$ is unique up to isomorphism. 
   \end{thm}
   \begin{rmk}
       Ribet's statement \cite{rib} are formulated using arithmetic Frobenius. We state  the Theorem using geometric Frobenius, which corresponds to taking the dual representation. 
   \end{rmk}
  \textbf{The proof of Theorem \ref{modular}:}
   Let $p\notmid lDM$ be a prime. 
   \\\indent Suppose that $p$ splits in $K$. Then $p=\mathfrak{p}\overline{\mathfrak{p}}$ for  prime  $\mathfrak{p}\subset K$. In particular, $p$ splits completely in $K$. \\\indent By Theorem \ref{hecke} and Theorem \ref{ribet},  we have 
   \begin{align}
       a_{p}=\widetilde{\psi}(\mathfrak{p})+\widetilde{\psi}(\overline{\mathfrak{p}})= \psi(u_{\mathfrak{p}})+\psi(\overline{u_{\mathfrak{p}}})  = {\rm Tr}(\rho_{l}(1)({\rm Frob}_{p})), \notag 
   \end{align}
       and 
       \begin{align}
           \varepsilon(p)p^{2}=& \eta(p)\varphi(p)p^{2}=\widetilde{\psi}(p) =\psi(u_{\mathfrak{p}})\psi(\overline{u_{\mathfrak{p}}})= {\rm det}(\rho_{l}({\rm Frob}_{p})). \notag 
       \end{align}
        Assume that $p$ is inert. By Proposition \ref{frob}, Theorem \ref{coef}, and Theorem \ref{hecke}, 
        \begin{align}
            a_{p}=0={\rm Tr}(\rho_{l}(1)({\rm Frob}_{p}), \notag 
        \end{align}
        and 
        \begin{align}
            \varepsilon(p)p^{2}=&\eta(p)\varphi(p)p^{2} = -\psi(p) \notag 
            = {\rm det}(\rho_{l}(1)({\rm Frob}_{p})). \notag 
        \end{align}
        By the Chebotarev's density theorem, we have
        \begin{align}
            \rho_{l}^{ss}\cong \rho_{f, l}(-1), \notag
        \end{align}
        where $\rho_{l}^{ss}$ denotes the semisimplification of $\rho_{l}$. Therefore, 
        \begin{align}
            L(\rho_{l}, s)=L(f, s-1). \notag 
        \end{align}
    
\end{document}